\documentclass[12pt,reqno]{amsart}

\usepackage{amssymb}
\usepackage{amsmath}
\usepackage{latexsym}
\usepackage{amsthm}
\usepackage{epsfig}
\usepackage{color}
\usepackage{amsfonts,amssymb,mathrsfs,amscd}

\theoremstyle{plain}
\newtheorem{theorem}{Theorem}[section]

\newtheorem{corollary}{Corollary}[section]
\newtheorem{lemma}{Lemma}[section]
\newtheorem{remark}{\bf Remark}[section]
\theoremstyle{definition}

\newcommand{\Tr}{\mathop\mathrm{Tr}}

\title
[Magnetic interpolation inequalities ]
 {One-dimensional interpolation inequalities,  Carlson--Landau
  inequalities and magnetic Schr\"odinger operators }

\author[A.Ilyin, A.Laptev, M.Loss, S.Zelik] {Alexei  Ilyin,
 Ari Laptev, Michael Loss, Sergey Zelik}


\address
{\noindent\newline  Keldysh Institute of Applied Mathematics and
Institute for Information Transmission Problems;\newline
Imperial College London and Institute Mittag--Leffler;
\newline School of Mathematics, Georgia Institute of
Technology;
\newline University of Surrey, Department of Mathematics and
Keldysh Institute of Applied Mathematics}
\email{ ilyin@keldysh.ru; a.laptev@imperial.ac.uk;
\newline loss@math.gatech.edu;
 s.zelik@surrey.ac.uk}

\begin{document}

\maketitle

\medskip

\bigskip
\begin{quote}
{\normalfont\fontsize{8}{10}
\selectfont{\bfseries  Abstract.}
In this paper we prove refined
 first-order
 interpolation inequalities for periodic functions
and give applications to
 various refinements  of the  Carlson--Landau-type inequalities
and  to magnetic Schr\"odinger operators. We also obtain Lieb-Thirring 
inequalities for magnetic Schr\"odinger operators on multi-dimensional cylinders.
}
\end{quote}

\setcounter{equation}{0}
\section{Introduction}
\label{sec1}

The Carlson inequality~\cite{Carl}
\begin{equation}\label{Carlorig}
\biggl(\,\sum_{k=1}^\infty a_k\biggr)^2\leqslant
\pi\biggl(\,\sum_{k=1}^\infty a_k^2\biggr)^{1/2}
\biggl(\,\sum_{k=1}^\infty k^2a_k^2\biggr)^{1/2}
\end{equation}
has been a source of many improvements, refinements and
generalizations (see~\cite{HLP},\cite{Carl_book} and the
references therein). The constant $\pi$ here is sharp and
the inequality is strict unless $\{a_k\}_{k=1}^\infty\equiv0$.

This inequality and its various generalizations are closely
connected with classical one-dimensional interpolation inequalities
for Sobolev spaces:
\begin{equation}\label{Sob}
\|u\|_\infty^2\le \mathrm{C}(m)\|u\|^{2\theta}\|u^{(m)}\|^{2(1-\theta)},
\quad\theta=1-\frac1{2m},\ m>\frac12.
\end{equation}

In the case when $x\in\mathbb{R}$ the sharp constant and the
corresponding extremals
 were found in \cite{Taikov}:
$$
 \mathrm{C}(m)=\frac1
{\theta^\theta(1-\theta)^{1-\theta}2m\sin\frac\pi{2m}}.
$$

In the periodic case $x\in(0,2\pi)$ with zero average condition the inequality holds with the same constant
(without extremal functions)~\cite{I98JLMS}. Furthermore, the first-order
inequality with $\mathrm{C}(1)=1$ is equivalent
(as was first observed  in~\cite{Hardy})
to~\eqref{Carlorig} by going over from $\{a_k\}_{k=1}^\infty$ to
$u(x)=\sum_{k\in\mathbb{Z}_0}a_{|k|}e^{ikx}$
and using Parseval's equality.
Here in what follows $\|\cdot\|_\infty:=\|\cdot\|_{L_\infty}$,
$\|\cdot\|:=\|\cdot\|_{L_2}$, and
 $\mathbb{Z}_0:=\mathbb{Z}\setminus\{0\}$.

For all $m>1/2$ inequality  admits a
negative correction term on the right-hand side~\cite{Zelik}, \cite{IZ},
in particular, in the first- and second-order cases
the correction term can be written in closed form
\begin{eqnarray}
\|u\|_{\infty}^2 &\le & \|u\|\,\|u'\|-\frac{1}{\pi}\|u\|^2, \label{rem}\\
\|u\|_{\infty}^2 &\le& \frac{\sqrt{2}}{\sqrt[4]{27}}|u\|\,\|u'\|
-\frac{2}{3\pi}\|u\|^2,\label{rem1}
\end{eqnarray}
where all constants are sharp and no extremals exist.
Again, for $u(x)=\sum_{k\in\mathbb{Z}_0}a_{|k|}e^{ikx}$
this gives the following two improved  Carlson inequalities
\begin{eqnarray}
\biggl(\,\sum_{k=1}^\infty a_k\biggr)^2 &\le & \pi\biggl(\,\sum_{k=1}^\infty a_k^2\biggr)^{1/2}
\biggl(\,\sum_{k=1}^\infty k^2a_k^2\biggr)^{1/2}
-\sum_{k=1}^\infty a_k^2, \label{Carl}\\
\biggl(\,\sum_{k=1}^\infty a_k\biggr)^2  &\le&
\frac{\sqrt{2}\pi}{\sqrt[4]{27}}\biggl(\,\sum_{k=1}^\infty a_k^2\biggr)^{3/4}
\biggl(\,\sum_{k=1}^\infty k^4a_k^2\biggr)^{1/4}
-\frac{2}{3}\sum_{k=1}^\infty a_k^2,\label{Carl3}
\end{eqnarray}
with sharp constants.
These inequalities are proved in \cite{Zelik},\cite{IZ}
in the framework of a rather general theory and we give
below in \S\,\ref{sec2}  a new direct self-contained proof
of~\eqref{rem} and \eqref{Carl}.

In \S\,\ref{sec3} we consider inequalities of the form
\begin{equation}\label{magn2}
\|u\|^2_\infty\le K(\alpha)\|A^{1/2}u\|\|u\|
\end{equation}
for $2\pi$-periodic functions (no zero average condition),
where
$$
\|A^{1/2}u\|^2=\int_0^{2\pi}
\left|i\,\frac{ du}{dx}-a(x) u\right|^2dx
$$
is the quadratic form corresponding to the Schr\"odinger
operator
$$
Au=\left(i\,\frac{ du}{dx}-a(x) u\right)^2
$$
with  magnetic potential $a\in L_1(0,2\pi)$.
The sharp constant $K(\alpha)$
depends only on the flux
$$
\alpha:=\frac1{2\pi}\int_0^{2\pi}a(x)dx,
$$
it is finite  if and only if $\alpha\notin\mathbb{Z}$.
In a  somewhat similar situation considered in \cite{Lap-Weid-Mag}
the introduction of a magnetic field has made it possible to
prove the Hardy inequality in $\mathbb{R}^2$. In our periodic case
a magnetic field with non-integral flux removes the condition
$\int_0^{2\pi}u(x)dx=0$.

The expression for $K(\alpha)$  is as follows
\begin{equation}\label{Kalpha1}
K(\alpha)=\left\{
            \begin{array}{ll}
            {|\sin(2\pi\alpha)|}^{-1}, & \hbox{$\alpha\operatorname{mod}(1)\in(0,1/4)\cup(3/4,1)$;} \\
              1, & \hbox{$\alpha\operatorname{mod}(1)\in[1/4,3/4]$.}
            \end{array}
          \right.
\end{equation}
In the first case there exists a unique extremal function
and for $\alpha\in[1/4,3/4]$ there are no extremals and a
negative correction term
may exist. We show in \S\,\ref{sec4}  that this is indeed the case and
\begin{equation}\label{magneticcor}
\aligned
\|u\|^2_\infty&\le\|A^{1/2}u\|\|u\|
\bigl(1-2e^{-4\pi\|A^{1/2}u\|/\|u\|}\bigr),\ \alpha=1/4,\
\alpha=3/4,\\
\|u\|^2_\infty&\le\|A^{1/2}u\|\|u\|
\bigl(1+2\cos(2\pi\alpha)e^{-2\pi\|A^{1/2}u\|/\|u\|}\bigr),\
\alpha\in(1/4,3/4).
\endaligned
\end{equation}

In \S\,\ref{sec5} we consider applications to  Carlson--Landau
inequalities.
The  Landau improvement
of~\eqref{Carlorig}
(see, for instance, ~\cite{HLP})
\begin{equation}
\label{Landau}
\biggl(\,\sum_{k=1}^\infty a_k\biggr)^2\leqslant
\pi\biggl(\,\sum_{k=1}^\infty a_k^2\biggr)^{1/2}
\biggl(\,\sum_{k=1}^\infty (k-1/2)^2a_k^2\biggr)^{1/2}
\end{equation}
 has a surprisingly short (and almost elementary)
proof in terms of our  interpolation inequalities. We recall
the elementary inequality (which is~\eqref{Sob} with $m=1$)
\begin{equation}\label{inter}
\|u\|_\infty^2\le \|u\|\|u'\|,
\qquad u\in H^1_0(0,L),
\end{equation}
following from
$$
2u(x)^2=\int_0^x(u(t)^2)'dt-\int_x^L(u(t)^2)'dt\le 2\|u\|\|u'\|.
$$
Given a (non-negative) sequence $\{a_k\}_{k=1}^\infty$ we
 set $L=1$ and consider the function
\begin{equation}\label{u(x)}
u(x)=\sqrt{2}\sum_{k=1}^\infty
(-1)^{k+1}a_k\sin(2k-1)\pi x,\quad x\in [0,1].
\end{equation}
We  have $\|u\|_\infty=u(\pi/2)=\sqrt{2}\sum_{k=1}^\infty
a_k$ and by orthonormality,
\begin{equation}\label{Parceval}
\|u\|^2=\sum_{k=1}^\infty a_k^2,
\qquad \|u'\|^2=\pi^2\sum_{k=1}^\infty(2k-1)^2 a_k^2=
4\pi^2\sum_{k=1}^\infty(k-1/2)^2 a_k^2.
\end{equation}
Substituting this into
\eqref{inter}  we obtain
inequality~\eqref{Landau}.

The refinement of~\eqref{inter} obtained in \cite{IZ}
\begin{equation}\label{inter1}
\|u\|_\infty^2\le \|u\|\|u'\|\bigl(1-2e^{-L\|u'\|/\|u\|}\bigr),
\qquad u\in H^1_0(0,L)
\end{equation}
or, equivalently, inequality~\eqref{magneticcor} in the symmetric case
 $\alpha=1/2$
give a sharp  correction term to the Carlson--Landau inequality
\eqref{Landau}
\begin{equation}\label{Expcorr}
\biggl(\,\sum_{k=1}^\infty a_k\biggr)^2\le
\pi\|a\|\|a\|_1\bigl(1-2e^{-2\pi\|a\|_1/\|a\|}\bigr).
\end{equation}
Next, using a second-order inequality in~\cite{IZ}
we obtain the following sharp inequality
$$
\biggl(\,\sum_{k=1}^\infty a_k\biggr)^2\le
\frac{\sqrt{2}\,\pi}{\sqrt[4]{27}}\coth(\pi/2)
\|a\|^{3/2}
\|a\|_2^{1/2}
$$
with unique extremal $a_k=1/(2k-1)^4+4)$.
Here we set for brevity

\begin{equation}\label{notation}
\|a\|^2=\sum_{k=1}^\infty a_k^2,\quad
\|a\|^2_1=\sum_{k=1}^\infty (k-1/2)^2a_k^2,\quad
\|a\|^2_2=\sum_{k=1}^\infty (k-1/2)^4a_k^2.
\end{equation}

The whole family of Carlson--Landau inequalities
\begin{equation}\label{interCLalpha}
\left(\sum_{k=1}^\infty a_k\right)^2\le k(\alpha)\left(\sum_{k=1}^\infty a_k^2\right)^{1/2}\left(\sum_{k=1}^\infty
(k-\alpha)^2a_k^2\right)^{1/2},
\end{equation}
is studied for $\alpha\in[0,1)$ in Theorem~\ref{Th:Intermediate}.
Obviously, $k(\alpha)=\pi$ for $\alpha\in[0,1/2]$ and, furthermore,
for $\alpha\in[0,1/2)$ we have a sharp $L_2$-type correction term here,
see~\eqref{interCL}.
In the symmetric case $\alpha=1/2$ the correction term is exponentially
small, see~\eqref{Expcorr}. For $\alpha\in(1/2,1)$ we show
that $k(\alpha)>\pi$, moreover, $k(\alpha)\sim(1-\alpha)^{-1}$
as $\alpha\to1^-$,
and there exists a unique extremal.

Finally, in \S\,\ref{sec6} we consider applications to the Lieb--Thirring
inequalities and first  give a new alternative proof
of the main result in~\cite{D-L-L} on the one-dimensional
Sobolev inequalities for orthonormal families of vector-functions
 along with generalizations to higher-order derivatives and
 1-D magnetic forms. This gives the Lieb--Thirring
estimate for the negative trace of a 1-D magnetic Schr\"odinger
operator with a matrix-valued potential.
Then we combine this result with the main ideas and results
in~\cite{AisLieb}, \cite{D-L-L}, \cite{H-L-W}, and \cite{Lap-Weid} to obtain in Theorem~\ref{Th gamma}
estimates for the $1/2$- and $1$- moments
of the negative eigenvalues of the Schr\"odinger-type
operator in $\mathbb{T}^{d_1}_x\times\mathbb{R}^{d_2}_y$.
For example, for $d_1=d_2=1$ and the operator
$$
\mathcal{H}\Psi=-\frac{d^2}{dy^2}\Psi+
\left(i\,\frac{ d}{dx}-a(x)\right)^2 \Psi-V(x,y)\Psi=-\lambda\Psi
$$
on the cylinder $\mathbb{R}_y\times\mathbb{S}^1_x$ we have the following estimates for its negative eigenvalues:
\begin{equation}\label{1/2-mom}
\aligned
\sum_{k}\lambda_k^{1/2}&\le
\frac1{3\sqrt{3}}K(\alpha)\int_{\mathbb{R}\times\mathbb{S}^1}
V^{3/2}(x,y)dydx,\\
\sum_{k}\lambda_k&\le
\frac1{8\sqrt{3}}K(\alpha)\int_{\mathbb{R}\times\mathbb{S}^1}
V^{2}(x,y)dydx.
\endaligned
\end{equation}

For $d_1=2$, $d_2=0$ and the operator
$$
\mathcal{H}\Psi=\left(i\,\frac{ d}{dx_1}-a_1(x_1)\right)^2\Psi+
\left(i\,\frac{ d}{dx_2}-a_2(x_2)\right)^2 \Psi-V(x_1,x_2)\Psi=-\lambda\Psi
$$
on the torus $\mathbb{T}^2$
with $\alpha_j=\frac1{2\pi}\int_0^{2\pi}a(x_j)dx\notin\mathbb{Z}$, $j=1,2$ we have
\begin{equation}\label{torus2}
\sum_{k}\lambda_k\le
\frac{\pi}{24}\,K(\alpha_1)K(\alpha_2)\int_{\mathbb{T}^2}
V^{2}(x_1,x_2)dx_1dx_2.
\end{equation}

Note  that in the region  where
$K(\alpha)=1$, the constants in \eqref{1/2-mom}
coincide with the best-known constants in the corresponding
Lieb--Thirring inequalities for the
Schrodinger operator in $\mathbb{R}^2$,
see \cite{D-L-L}, \cite{H-L-W}. However, the constant in
\eqref{torus2} contains an extra factor $\pi/\sqrt{3}$,
since when we apply ``the lifting argument with respect
 to dimensions'' \cite{Lap-Weid} in the direction $x$, we do not have semiclassical estimates for the $\gamma$-Riesz means with $\gamma\ge 3/2$ for the negative eigenvalues
in the periodic case. This factor along with $K(\alpha_j)$ accumulates with each
iteration of the lifting procedure with respect to the
$x$-variables, see Theorem~\ref{Th gamma}.

\section{Proof of first-order inequality}\label{sec2}

We consider the following maximization
problem for $2\pi$-periodic functions with zero average: for $D\ge1$  find $\mathbb{V}(D)$ -- the solution of
the following extremal problem
\begin{equation}
\label{V}
\mathbb V(D):=\sup\bigl\{|u(0)|^2, \
\|u\|^2=1, \ \|u'\|^2=D\bigr\}.
\end{equation}

The next lemma  gives an implicit formula for
the function $\mathbb V(D)$.
\begin{lemma}\label{L:impl} The following expression holds for $\mathbb V(D)$:
\begin{equation}
\label{Vexp}
\mathbb V(D)=\frac1{2\pi}\,
\frac{(\sum_{k\in\mathbb{Z}_0}\frac1{\lambda+k^2})^2}
{\sum_{k\in\mathbb{Z}_0}\frac1{(\lambda+k^2)^2}}\,,
\end{equation}
where $\lambda=\lambda(D)$ is a unique solution
of the functional equation
\begin{equation}\label{eqforlam}
\frac{\sum_{k\in\mathbb{Z}_0}\frac{k^2}{(\lambda+k^2)^2}}
{\sum_{k\in\mathbb{Z}_0}\frac1{(\lambda+k^2)^2}}=D\ (=:D(\lambda)).
\end{equation}
Furthermore, $\lambda(1)=-1$ and $\lambda(\infty)=\infty$.

\end{lemma}
\begin{proof}
Using the Fourier series
$u(x)=\sum_{k\in\mathbb{Z}_0}u_k e^{ikx}$
and the Parseval equalities
$\|u\|^2=2\pi\sum_{k\in\mathbb{Z}_0}|u_k|^2$,
$\|u'\|^2=2\pi\sum_{k\in\mathbb{Z}_0}k^2|u_k|^2$,
for every $\lambda>-1$ we have by the Cauchy--Schwartz
inequality
\begin{equation}\label{freelambda}
\aligned
|u(0)|^2=\biggl|\sum_{k\in\mathbb{Z}_0}u_k\biggr|^2\le
\left(\sum_{k\in\mathbb{Z}_0}|u_k|\right)^2=\\=
\left(\sum_{k\in\mathbb{Z}_0}|u_k|(\lambda+k^2)^{1/2}(\lambda+k^2)^{-1/2}\right)^2\le\\\le
\sum_{k\in\mathbb{Z}_0}|u_k|^2(\lambda+k^2)
\sum_{k\in\mathbb{Z}_0}\frac1{\lambda+k^2}=\\=
\frac1{2\pi}(\lambda\|u\|^2+\|u'\|^2)
\sum_{k\in\mathbb{Z}_0}\frac1{\lambda+k^2}=\\=
\frac1{2\pi}\|u\|^2(\lambda+\|u'\|^2/\|u\|^2)
\sum_{k\in\mathbb{Z}_0}\frac1{\lambda+k^2}\,.
\endaligned
\end{equation}
Moreover, for (and only for)
$$
u_k=u_k^{(\lambda)}:=\mathrm{const}\frac1{\lambda+k^2}
$$
the above inequalities turn into equalities.
Fixing for definiteness $\mathrm{const}:=\frac1{2\pi}$ we consider the function
$
G_\lambda(x)=\frac1{2\pi}
\sum_{k\in\mathbb{Z}_0}\frac{e^{ikx}}{\lambda+k^2}\,,
$
for which
$$
\frac{\|(G_\lambda)'_x\|^2}{\|G_\lambda\|^2}=
\frac{\sum_{k\in\mathbb{Z}_0}\frac{k^2}{(\lambda+k^2)^2}}
{\sum_{k\in\mathbb{Z}_0}\frac1{(\lambda+k^2)^2}}
=:D(\lambda).
$$
(We also observe that $G_\lambda$ solves the
equation
$
-\frac{d^2}{dx^2}G_\lambda+\lambda G_\lambda=\delta
$.)

Then we see that for every fixed $\lambda>-1$
the normalized function
$
{G_\lambda(x)}/{\|G_\lambda\|}
$
is the extremal function in problem \eqref{V} with
$D=D(\lambda)$, and the value of the solution
function $\mathbb{V}(D)$ is
$$
\aligned
\mathbb{V}(D)=\frac1{2\pi}(\lambda+D(\lambda))
\sum_{k\in\mathbb{Z}_0}\frac1{\lambda+k^2}=\\=
\frac1{2\pi}\left(\lambda+\frac{\sum_{k\in\mathbb{Z}_0}\frac{k^2}{(\lambda+k^2)^2}}
{\sum_{k\in\mathbb{Z}_0}\frac1{(\lambda+k^2)^2}}\right)
\sum_{k\in\mathbb{Z}_0}\frac1{\lambda+k^2}=
\frac1{2\pi}\,
\frac{(\sum_{k\in\mathbb{Z}_0}\frac1{\lambda+k^2})^2}
{\sum_{k\in\mathbb{Z}_0}\frac1{(\lambda+k^2)^2}}\,.
\endaligned
$$

We now have  to show that there exists a unique
$\lambda=\lambda(D)$ solving~\eqref{eqforlam} for
every fixed $D\ge1$. We first observe that
$D(\lambda)\to1$ as $\lambda\to-1$ and $D(\lambda)\to\infty$
as $\lambda\to\infty$.
 It remains to show that $D(\lambda)$ is strictly monotone
increasing.
We have
$$
\aligned
\frac{d}{d\lambda}D(\lambda)=\frac2
{(\sum_{k\in\mathbb{Z}_0}\frac1{(\lambda+k^2)^2})^2}\cdot
\sum_{k,l\in\mathbb{Z}_0}
\frac{k^2(k^2-l^2)}{(\lambda+k^2)^3(\lambda+l^2)^3}=\\=
\frac2
{(\sum_{k\in\mathbb{Z}_0}\frac1{(\lambda+k^2)^2})^2}\cdot
\sum_{k,l\in\mathbb{Z}_0,k>l}
\frac{k^2(k^2-l^2)+l^2(l^2-k^2)}{(\lambda+k^2)^3(\lambda+l^2)^3}=\\=
\frac2
{(\sum_{k\in\mathbb{Z}_0}\frac1{(\lambda+k^2)^2})^2}\cdot
\sum_{k,l\in\mathbb{Z}_0,k>l}
\frac{(k^2-l^2)^2}{(\lambda+k^2)^3(\lambda+l^2)^3}\,>\,0.
\endaligned
$$
The proof is complete.
\end{proof}

We set
\begin{equation}\label{fgh}
G(\lambda):=G_\lambda(0)=\frac1{2\pi}
\sum_{k\in\mathbb{Z}_0}\frac1{\lambda+k^2}\,.
\end{equation}

The following variational characterization of
$\mathbb{V}(D)$ is important.
\begin{theorem}\label{T:Var}
For a fixed $D\ge1$
\begin{equation}\label{Var}
\mathbb{V}(D)=\min_{\lambda\in[-1,\infty)}(\lambda+D)G(\lambda).
\end{equation}
\end{theorem}
\begin{proof}
Since $G(\lambda)\to+\infty$ as $\lambda\to-1$ and
$G(\lambda)=O(\lambda^{-1/2})$ as $\lambda\to+\infty$,
it follows that the minimum is attained for each fixed $D\ge1$ at some point $\lambda_*=\lambda_*(D)$.
Then
$$
\frac{d}{d\lambda}\bigl((\lambda+D)G(\lambda)\bigr)\vert_{\lambda=\lambda_*}=0,
$$
which gives
$$
D=-\frac{G(\lambda)}{G'(\lambda)}-\lambda.
$$
In view of~\eqref{fgh} this equation coincides with \eqref{eqforlam} and
therefore $\lambda_*(D)$ coincides with the unique
inverse function $\lambda(D)$ constructed in
Lemma~\ref{L:impl}.
This gives
$$
(\lambda(D)+D)G(\lambda(D))=-\frac{G(\lambda(D))^2}
{G'(\lambda(D))}=\mathbb{V}(D).
$$
The proof is complete.
\end{proof}

Of course, it is impossible to find an explicit formula for the inverse function $\lambda=\lambda(D)$, therefore
it is impossible to find an explicit formula for
$\mathbb{V}(D)$. However, it is possible to find the
asymptotic expansion of $\mathbb{V}(D)$ as $D\to \infty$.
All that we need to know for this purpose is the asymptotic
expansion of the function $G(\lambda)$ as $\lambda\to \infty$. This  expansion, in turn, is
found by the Poisson summation formula
(or by means of the explicit formula \eqref{fexpl}):
\begin{equation}\label{asympf}
G(\lambda)=\frac1{2\pi}
\sum_{k\in\mathbb{Z}_0}\frac1{\lambda+k^2}=
\frac12\lambda^{-1/2}-\frac1{2\pi}\lambda^{-1}
+O(e^{-\pi\lambda^{1/2}})
\quad\text{as}\ \lambda\to\infty.
\end{equation}
\begin{lemma}\label{L:AsymV}
It holds as $D\to\infty$
\begin{equation}\label{asympV}
\mathbb{V}(D)=D^{1/2}-\frac1{\pi}-\frac1{2\pi^2}D^{-1/2}+
O(D^{-1}).
\end{equation}
\end{lemma}
\begin{proof}
This is a particular case of the general result
of Proposition~2.1 in \cite{IZ}. In addition to~\eqref{asympf}
 we have
\begin{equation}\label{asympgh}
G'(\lambda)=-\frac14\lambda^{-3/2}+\frac1{2\pi}\lambda^{-2}
+O(\lambda^{-5/2}),
\end{equation}
and, hence,
\begin{equation}\label{asympD(lambda)}
D(\lambda)=-\frac{G(\lambda)}{G'(\lambda)}-\lambda=\lambda+\frac2\pi\lambda^{1/2}+
\frac4{\pi^2}+O(\lambda^{-1/2}).
\end{equation}
The well-defined inverse function $\lambda(D)$ (see~\eqref{eqforlam})
has the asymptotic behaviour
\begin{equation}\label{asymplambda(D)}
\lambda(D)=D-\frac2\pi D^{1/2}-\frac2{\pi^2}+O(D^{-1/2})
\quad\text{as}\ D\to\infty.
\end{equation}
Substituting this into~\eqref{asympf}, \eqref{asympgh},
we obtain for $\mathbb{V}(D)=-\frac{G^2(\lambda(D))}{G'(\lambda(D))}$
the asymptotic expansion~\eqref{asympV}.
The proof is complete.
\end{proof}

The third term in \eqref{asympV} is negative,
hence,
\begin{equation}\label{main}
\mathbb{V}(D)<D^{1/2}-\frac1{\pi}
\end{equation}
for all sufficiently large $D\ge D_0$.
Therefore we shall have proved inequality
\eqref{rem} once we have shown that \eqref{main}
holds for \textit{all} $D\ge1$.
Moreover, Lemma~\ref{L:AsymV} implies that both
constants in \eqref{rem} are sharp.
\begin{theorem}\label{T:est}
Inequality~\eqref{main} holds for all $D\ge1$.
\end{theorem}
\begin{corollary}\label{Cor:2}
Inequalities \eqref{rem} and \eqref{Carl}
hold and all the constants there are sharp.
\end{corollary}
\begin{proof}
By homogeneity and \eqref{main}, for a $u\in\dot H^1$
$$
u(0)^2\le \|u\|^2\mathbb{V}\left(\frac{\|u'\|^2}{\|u\|^2}\right)<
\|u\|^2\cdot\frac{\|u'\|}{\|u\|}-\frac 1\pi \|u\|^2.
$$

\end{proof}
\begin{proof}[Proof of the theorem]
The proof is based on the variational representation
\eqref{Var} and the
explicit formula for $G(\lambda)$:
\begin{equation}\label{fexpl}
G(\lambda)=\frac1{2\pi}
\sum_{k\in\mathbb{Z}_0}\frac1{\lambda+k^2}=
\frac{1}{2\pi}\,
\frac{\pi\sqrt\lambda\,\coth(\pi\sqrt\lambda\,)-1}{\lambda}.
\end{equation}
We estimate $G(\lambda)$ by a more convenient expression
\begin{equation}\label{convi}
G(\lambda)<\frac{\pi\sqrt{\lambda}-1+e^{-\pi\sqrt{\lambda}}}
{2\pi\lambda}=:G_0(\lambda),
\end{equation}
where the above  inequality  by equivalent transformations reduces to $x<\sinh(x)$, $x>0$.

 Thus, in view of~\eqref{Var} and  \eqref{convi}, for $D\ge1$
$$
\aligned
\mathbb V(D)
\le
(\lambda+D)G_0(\lambda)\big|_{\lambda=(D^{1/2}-1/2)^2}=:
\mathbb V_0(y(D)),
\endaligned
$$
where $y=y(D):=D^{1/2}-1/2$, $y\ge 1/2$ and
$$
\mathbb V_0(y)=\frac1{2\pi y^2}\bigl(\pi y-1+e^{-\pi y}\bigr)
\bigl(y^2+(y+1/2)^2\bigr).
$$
Now
\begin{equation}\label{1}
\aligned
\mathbb{V}(D)-D^{1/2}+\frac1{\pi}<
\mathbb{V}_0(y)-\left(y+\frac12\right)+\frac1{\pi}=\\=
\frac1{8\pi y^2}\left((8y^2+4y+1)e^{-\pi y}-(4-\pi)y-1\right)
=:\frac1{8\pi y^2}W(y).
\endaligned
\end{equation}
Next,
$$
W'(y)=\bigl(-8\pi y^2+(16-4\pi)y+4-\pi\bigr)e^{-\pi y}-4+\pi
$$
and the coefficient of $e^{-\pi y}$ is negative for $y\ge1/2$.
Therefore $W'(y)<0$ and  $W(y)$ is decreasing for $y\ge1/2$ and
\begin{equation}\label{2}
W(y)\le W(1/2)=5e^{-\pi/2}-3+\pi/2=-0.3898<0,
\end{equation}
which completes the proof of~\eqref{main}.
\end{proof}
\begin{remark}
{\rm
 The proof of inequality~\eqref{main} in the last theorem
is in the spirit of Hardy's first proof~\cite{Hardy} of
the original Carlson inequality~\eqref{Carlorig} and is, in fact,
self-contained and formally independent of the previous
argument. It follows from~\eqref{freelambda} that
$$
\mathbb{V}(D)\le (\lambda+D)G(\lambda),
$$
where $\lambda\ge-1$ is an arbitrary free parameter.
Therefore inequality~\eqref{main} will be proved if we succeed
in finding such a substitution $\lambda=\lambda(D)$ for
which
$$
(\lambda(D)+D)G(\lambda(D))<D^{1/2}-\frac1\pi\,
\quad\text{for all $D\ge1$}.
$$
Now estimates \eqref{1} and \eqref{2} in the proof of
Theorem~\ref{T:est} are saying that the
substitution $\lambda(D)=D-D^{1/2}+1/4$ will do the job.
This substitution  agrees in the leading term
with~\eqref{asymplambda(D)}. The lower order terms are
`experimental'. Also, without knowing~\eqref{asymplambda(D)}
finding this substitution becomes much more difficult.

On the other hand, the proof of sharpness is contained in Lemma~\ref{L:AsymV}.
 Alternatively, we can verify sharpness of \eqref{rem}
(and~\eqref{Carl}) at the test
 function $\sum_{k\in\mathbb{Z}_0}\frac{e^{ikx}}{\lambda+k^2}$
 by letting $\lambda\to\infty$.
}

\end{remark}

\setcounter{equation}{0}
\section{Magnetic inequality}
\label{sec3}

We are interested in the inequality
\begin{equation}\label{magnetic}
\|u\|_\infty^2\le K(\alpha)\left(\int_0^{2\pi}
\left|i\,\frac{ du}{dx}-a(x) u\right|^2dx\right)^{1/2}
\left(\int_0^{2\pi}
|u(x)|^2dx\right)^{1/2},
\end{equation}
where $u$ is a $2\pi$-periodic function
(which may be a constant so no zero-mean
condition is assumed), and $a\in L^1(0,2\pi)$. Here $K(\alpha)$ denotes a sharp constant and we show below that it depends only  on the flux
\begin{equation}\label{flux}
\alpha:=\frac1{2\pi}\int_0^{2\pi}a(x)dx,
\end{equation}
and $K(\alpha)<\infty$ if and only if $\alpha\notin\mathbb{Z}$.

\subsection*{Constant magnetic potential.}
We first consider the case when
$a(x)\equiv \alpha\in(0,1)$.
Setting
$$
A=\left(i\,\frac d{dx}-\alpha\right)^2
$$
we consider the positive-definite self-adjoint operator
$$
\mathbb{A}(\lambda):=A+\lambda I,
\qquad\lambda\ge -\min(\alpha^2,(1-\alpha)^2)
$$
and its Green's function $G_\lambda(x,\xi)$:
$$
\mathbb{A}(\lambda)G_\lambda(x,\xi)=\delta(x-\xi),
$$
 which is found in terms of the Fourier series
$$
G_\lambda(x,\xi)=\frac1{2\pi}\sum_{n\in\mathbb{Z}}
\frac{e^{in(x-\xi)}}{(n+\alpha)^2+\lambda}\,,
$$
so that
\begin{equation}\label{G}
G_\lambda(\xi,\xi)=\frac1{2\pi}\sum_{n\in\mathbb{Z}}
\frac{1}{(n+\alpha)^2+\lambda}=:G(\lambda)\,,
\end{equation}

The series can be summed explicitly
(for instance, by the Poisson summation formula)
\begin{equation}\label{Green}
G(\lambda)=\frac1{2\sqrt{\lambda}}\cdot
\frac{\sinh(2\pi\sqrt{\lambda})}
{\cosh(2\pi\sqrt{\lambda})-\cos(2\pi\alpha)}\,.
\end{equation}

By Theorem~2.2 in~\cite{IZ} with $\theta=1/2$
(see also Remark~\ref{R:direct})
\begin{equation}\label{Kalpha}
\aligned
K(\alpha)=
\frac{1}{\theta^\theta(1-\theta)^{1-\theta}}\cdot
\sup_{\lambda>0}\lambda^{\theta}G(\lambda)=
2\sup_{\lambda>0}\sqrt{\lambda} G(\lambda)=
\sup_{\varphi>0}F(\varphi),
\endaligned
\end{equation}
where
$$
F(\varphi):=\frac{\sinh\varphi}
{\cosh\varphi-\cos(2\pi\alpha)}
$$
 and $\varphi= 2\pi\sqrt{\lambda}$. Next, the derivative
$$
\frac d{d\varphi}F(\varphi)=
\frac{1-\cosh\varphi\cos(2\pi\alpha)}
{(\cosh\varphi-\cos(2\pi\alpha))^2}>0
$$
if $\cos(2\pi\alpha)\le0$, that is, if $\alpha\in [1/4,3/4]$,
so that in this case $F$ is increasing and  the supremum is  `attained' at infinity,
which gives
$$
K(\alpha)=1\qquad\text{for}\quad \frac14\le\alpha\le\frac34.
$$
Otherwise, for $\alpha\in(0,1/4)\cup(3/4,1)$ the function $F(\varphi)$
attains a global maximum  at
$$
\varphi_*(\alpha)=\operatorname{arccosh}
\left(\frac1{\cos(2\pi\alpha)}\right)\,,
$$
which gives
$$
K(\alpha)=F(\varphi_*(\alpha))=\frac1{|\sin(2\pi\alpha)|}\,.
$$
Finally, it is clear from the argument as well as from the result that it is $\alpha\operatorname{mod}(1)$ that really matters.

\subsection*{Non-constant magnetic potential.} Now
\begin{equation}\label{A}
A=\left(i\frac d{dx}-a(x)\right)^2,
\end{equation}
 and for the flux $\alpha$ defined in~\eqref{flux} let
$$
\varphi_n(x)=\frac1{\sqrt{2\pi}}
e^{i\bigl(n+\alpha)x-\int_0^xa(y)dy\bigr)}\,.
$$
Then $\{\varphi_n\}_{n=-\infty}^\infty$ is an orthonormal system in $L_2(0,2\pi)$.
Note that since $n\in\mathbb{Z}$,
these functions are periodic and also satisfy the equation
$$
\left(i\frac d{dx}-a(x)\right)\varphi_n=-(n+\alpha)\varphi_n
$$
and therefore we also have
$$
A\varphi_n=(n+\alpha)^2\varphi_n.
$$
In addition,  the system $\{\varphi_n\}_{n=-\infty}^\infty$
is complete (since  $\varphi_n(x)=c(x)e^{-inx}$ with
 $|c(x)|=1/\sqrt{2\pi}$). Then the Green's function for the operator
$A+\lambda I$ equals
$$
G_\lambda(x,\xi)=\sum_{n\in\mathbb{Z}}
\frac{\varphi_n(x-\xi)}{(n+\alpha)^2+\lambda}=
\frac1{2\pi}\sum_{n\in\mathbb{Z}}
\frac{e^{i\bigl(n+\alpha)(x-\xi)-\int_\xi^xa(y)dy\bigr)}}{(n+\alpha)^2+\lambda}\,
$$
and the expression for $G_\lambda(\xi,\xi)$ is exactly the
same as in~\eqref{G} and therefore everything after~\eqref{G}
is the same as in the case of a constant magnetic potential.

Thus, we have proved the following result.

\begin{theorem}\label{Th:magnetic}
Inequality~\eqref{magnetic}
holds for $\alpha\notin\mathbb{Z}$ and
the sharp constant $K(\alpha)$ is given by~\eqref{Kalpha1}.
Furthermore,  for $\alpha\operatorname{mod}(1)\in(0,1/4)\cup(3/4,1)$
there exists a unique extremal function
\begin{equation}\label{uextrem}
u_\lambda(x)=
\sum_{n\in\mathbb{Z}}
\frac{\varphi_n(x)}{(n+\alpha)^2+\lambda}=\frac1{\sqrt{2\pi}}
\sum_{n\in\mathbb{Z}}
\frac{e^{i\bigl(n+\alpha)x-\int_0^xa(y)dy\bigr)}}{(n+\alpha)^2+\lambda}\,,
\end{equation}
where
$$
\lambda=\lambda(\alpha):=\left[\frac1{2\pi}
\operatorname{arccosh}\left(\frac1{\cos(2\pi\alpha)}\right)\right]^2.
$$
There are no extremals for $\alpha\operatorname{mod}(1)\in[1/4,3/4]$.
\end{theorem}
\begin{remark}\label{R:direct}
{\rm In our one-dimensional case and operators
with explicitly known spectrum and eigenfunctions
 it makes sense to give a
direct proof of~\eqref{Kalpha}. In fact, using the Fourier
series $u(x)=\sum_{k\in\mathbb{Z}}u_k\varphi_k(x)$ and without loss of generality assuming that
$u(x)$ attains its maximum at $x=0$ we have
for an arbitrary $\lambda>0$ the following inequality
$$
\aligned
|u(0)|^2=\frac1{2\pi}\biggl|\sum_{k\in\mathbb{Z}}u_k\biggr|^2=\\=
\frac1{2\pi}\biggl|\sum_{k\in\mathbb{Z}}u_k((k+\alpha)^2+\lambda)^{1/2}
((k+\alpha)^2+\lambda)^{-1/2}\biggr|^2\le\\\le
\frac1{2\pi}\sum_{k\in\mathbb{Z}}\frac1{(k+\alpha)^2+\lambda)}
\sum_{k\in\mathbb{Z}}(|u_k|^2((k+\alpha)^2+\lambda))=\\=
G(\lambda)(\|A^{1/2}u\|^2+\lambda\|u\|^2),
\endaligned
$$
which turns into equality for $u(x)$ as in
\eqref{uextrem}.
For $\lambda_*=\|A^{1/2}u\|^2/\|u\|^2$
we see that
$$\|A^{1/2}u\|^2+\lambda_*\|u\|^2=2\lambda_*^{1/2}\
\|A^{1/2}u\|\|u\|$$ and therefore
$$
\|u\|^2_\infty\le2\sup_{\lambda>0}\lambda^{1/2}G(\lambda)
\|A^{1/2}u\|\|u\|,
$$
which shows that $K(\alpha)\le2\sup_{\lambda>0}\lambda^{1/2}G(\lambda)$.
To see that we have equality here,
we first assume that the supremum is attained at a finite point $\lambda_*<\infty$. Then
$$
\aligned
\sum_{k\in\mathbb{Z}}\frac1{(k+\alpha)^2+\lambda_*}&=
2\lambda_*\sum_{k\in\mathbb{Z}}\frac1{\bigl((k+\alpha)^2+\lambda_*\bigr)^2}
\ \left[=2\lambda_*\|u_{\lambda_*}\|^2\right],\\
\sum_{k\in\mathbb{Z}}\frac1{(k+\alpha)^2+\lambda_*}&=
2\sum_{k\in\mathbb{Z}}\frac{(k+\alpha)^2}{\bigl((k+\alpha)^2+\lambda_*\bigr)^2}
\ \left[=2\|A^{1/2}u_{\lambda_*}\|^2\right],
\endaligned
$$
where the first equality is
$(\lambda^{1/2}G(\lambda))'_{\lambda=\lambda_*}=0$, and
the validity of the second follows from the fact that
 the sum of the two equalities is a valid identity.
Since the left-hand side is equal to ${\sqrt{2\pi}}\|u_{\lambda_*}\|_\infty$
and $\lambda_*=\|A^{1/2}u_{\lambda_*}\|^2/\|u_{\lambda_*}\|^2$, recalling~\eqref{G}
we obtain
$$
\aligned
&\|u_{\lambda_*}\|^2_\infty=
\frac{1}{2\pi}\left(\sum_{k\in\mathbb{Z}}
\frac1{(k+\alpha)^2+\lambda_*}\right)^2=
{2\lambda_*}\|u_{\lambda_*}\|^2 G(\lambda_*)=
\\=
&2\lambda_*^{1/2} G(\lambda_*)
\lambda_*^{1/2}\|u_{\lambda_*}\|^2=
2\left(\lambda_*^{1/2}G(\lambda_*)\right)
\|A^{1/2}u_{\lambda_*}\|\|u_{\lambda_*}\|.
\endaligned
$$
This proves that $K(\alpha)=2\sup_{\lambda>0}\lambda^{1/2}G(\lambda)$
if $\lambda_*<\infty$.
Now we look at the case when $\lambda_*=\infty$.
Let $2\lim_{\lambda\to\infty}\lambda^{1/2}G(\lambda)=K'\ge K(\alpha)$. Setting
$H_N(\lambda)=2\lambda^{1/2}\sum_{|n|\le N}\frac1{(n+\alpha)^2+\lambda}$ we see that there
exists a sequence $N(j)\to\infty$ and a sequence $\lambda(j)\to\infty$ such that
$H_{N(j)}(\lambda(j))\to K'$. Since
$H_{N}(0)=H_{N}(\infty)=0$, it follows that
$H_{N(j)}(\lambda)$ attains a maximum at a $\lambda_*(j)<\infty$. The previous argument shows that
$H_{N(j)}(\lambda_*(j))$ is the sharp constant in our
inequality restricted to
$\mathrm{Span}\left\{\varphi_n\right\}_{n=-N(j)}^{N(j)}$.
 Therefore
$$
K(\alpha)\ge\limsup_{j\to\infty}H_{N(j)}(\lambda_*(j))\ge
\lim_{j\to\infty}H_{N(j)}(\lambda(j))=K'.
$$

As we have seen both cases are possible depending on whether
$\alpha\in[1/4,3/4]$  or $\alpha\in(0,1/4)\cup(3/4,1)$.
}
\end{remark}

\setcounter{equation}{0}
\section{Correction term}
\label{sec4} In the region $\alpha\in[1/4,3/4]$ no extremals exist
and therefore the might be a correction term in~\eqref{magn2}. By
symmetry the cases $\alpha$ and $1-\alpha$ are identical, therefore
we can and shall assume that $\alpha\in(0,1/2]$.  We now show that
the correction term indeed exists. We consider the maximization
problem
\begin{equation}\label{vD14}
\mathbb V(D):=\sup\{|u(0)|^2\colon
\ \| u\|^2=1, \
\|A^{1/2}\|^2=D\}, \ D\ge\alpha^2.
\end{equation}
Similarly to Theorem~\ref{T:Var} (see also the
general result in Theorem~2.3 in \cite{IZ})
we have
\begin{equation}\label{Vdmin}
\mathbb V(D)=\min_{\lambda\ge-\alpha^2}G(\lambda)(\lambda+D).
\end{equation}
We first consider  the cases $\alpha=1/2$
and $\alpha=1/4$.
We recall the elementary inequality~\eqref{inter}
and its refinement~\eqref{inter1} obtained in~\cite{IZ}
 and show that the case $\alpha =1/2$ or $1/4$
essentially reduces to the proof  of \eqref{inter1}
in~\cite{IZ}. In fact, for $\alpha=1/4$
 the key function \eqref{Green} becomes
$$
G(\lambda)=\frac1{2\sqrt{\lambda}}\tanh(2\pi\sqrt{\lambda}),
\quad\alpha=\frac14,
$$
and therefore
\begin{equation*}\label{vD14444}
\mathbb V(D)=\min_{\lambda\ge-1/16}G(\lambda)(\lambda+D)\le
\min_{\lambda\ge0}\frac1{2\sqrt{\lambda}}\tanh(2\pi\sqrt{\lambda})
(\lambda+D).
\end{equation*}
Up to a constant factor in the argument of $\tanh$
the minimum on the right-hand side was estimated
in \cite{IZ} (see (3.103), (3.104) there),  where it was shown that
\begin{equation}\label{prev}
\min_{\lambda\ge0}\frac{1}{2\sqrt{\lambda}}
\tanh\frac{\lambda^{1/2}}2(\lambda+D)<
\sqrt{D}\,\bigl(1-2e^{-\sqrt{D}}\bigr)
\end{equation}
Setting here $\mu=16\pi^2\lambda$ we obtain
for $\mathbb V(D)$ in \eqref{vD14} with $\alpha=1/4$
\begin{equation}\label{corr}
\mathbb V(D)<\sqrt{D}(1-2e^{-4\pi\sqrt{D}}).
\end{equation}
The case $\alpha=1/2$ is similar. Now in~\eqref{Green}
we have
$$
G(\lambda)=\frac1{2\sqrt{\lambda}}\tanh(\pi\sqrt{\lambda}),
\quad\alpha=\frac12,
$$
and in a totally similar way we find
$$
\aligned
\mathbb V(D)\le\min_{\lambda\ge0}
\frac1{2\sqrt{\lambda}}\tanh(\pi\sqrt{\lambda})(\lambda+D)
<
\sqrt{D}\left(1-2e^{-2\pi\sqrt{D}}\right).
\endaligned
$$

Thus, we have proved the following inequalities
\begin{eqnarray}
  \|u\|^2_\infty &\le& \|A^{1/2}u\|\|u\|
(1-2e^{-4\pi\|A^{1/2}u\|/\|u\|}),\quad\alpha=1/4,3/4, \label{magn33} \\
\|u\|^2_\infty &\le & \|A^{1/2}u\|\|u\|
(1-2e^{-2\pi\|A^{1/2}u\|/\|u\|}),\quad\alpha=1/2.\label{magn3}
\end{eqnarray}

The  case $\alpha\in(1/4,3/4)$ can be treated using
the general method of~\cite{IZ}. Our goal is to prove the
inequality
\begin{equation}\label{magnalpha}
\|u\|^2_\infty\le \|A^{1/2}u\|\|u\|
(1+2\cos (2\pi\alpha) e^{-2\pi\|A^{1/2}u\|/\|u\|}),\quad
1/4<\alpha<3/4,
\end{equation}
which is equivalent to
\begin{equation}\label{magnalphaV}
\mathbb{V}(D)\le\sqrt{D}(1+2\cos(2\pi\alpha)e^{-2\pi\sqrt{D}}),
\quad D\ge\alpha^2.
\end{equation}
In view of \eqref{Vdmin},  to prove this inequality it suffices to find such a substitutiuon
$\lambda=\lambda_*(D)$ for which
\begin{equation}
\label{subs}
 G(\lambda_*(D))
(\lambda_*(D)+D)\le\sqrt{D}(1+2\cos(2\pi\alpha)e^{-2\pi\sqrt{D}}),
\quad D\ge\alpha^2.
\end{equation}
The exact solution $\lambda=\lambda(D)$ for the minimizer, that is,
\begin{equation}\label{argmin}
\lambda(D)=\mathrm{argmin}\{G(\lambda)(\lambda+D)\}
\end{equation}
is the inverse function to the function $D=D(\lambda)$
$$
D(\lambda)=-\frac{G(\lambda)}{G'(\lambda)}-\lambda.
$$
It is impossible to find $\lambda(D)$ explicitly.
However, using~\eqref{Green} we can find the asymptotic expansion
$$
D(\lambda)=\lambda-4\pi a\lambda^{3/2}e^{-2\pi\sqrt{\lambda}}+
O(e^{-4\pi\sqrt{\lambda}})\quad\text{as}\ \lambda\to\infty,
$$
where
$$
a:=2\cos(2\pi\alpha).
$$
Therefore the inverse function $\lambda(D)$ (see~\eqref{argmin})
has the asymptotic behavior
$$
\lambda(D)=D+
4\pi a D^{3/2}e^{-2\pi\sqrt{D}}+
O(e^{-4\pi\sqrt{D}})\quad\text{as}\ D
\to\infty,
$$
truncating which we set
$$
\lambda_*(D)=D(1+
4\pi a \sqrt{D}e^{-2\pi\sqrt{D}}).
$$
Returning to~\eqref{magnalphaV}
 we find that
$$
G(\lambda_*(D))(\lambda_*(D)+D)=\sqrt{D}\,\Phi(y),
$$
where $y:=e^{-2\pi\sqrt{D}}$ and
$$
\Phi(y):=
\frac{1-ay\log y }{\sqrt{1-2ay\log y}}\cdot
\frac{1-y^{2\sqrt{1-2ay\log y}}}{1+y^{2\sqrt{1-2ay\log y}}
-ay^{\sqrt{1-2ay\log y}}}
\, .
$$
Therefore inequality~\eqref{subs} is equivalent to
\begin{equation}\label{subs1}
\Phi(y)<1+ay
\end{equation}
 for $y\in[0,e^{-2\pi\alpha}]$. The function $\Phi(y)$ has the asymptotic expansion
$$
\Phi(y)=1+ay+(-a^2(\log y)^2/2-2+a^2)y^2+O(y^3)
\quad \text{as}\ y\to 0^+,
$$
in which the coefficient of the quadratic term is negative for
all sufficiently small $y$. Therefore inequality~\eqref{subs1} holds for
all sufficiently small $y\in[0,y_0]$. The graphs of the
function
$\Phi(y)-(1+ay)$  on the whole intervals
$y\in[0,e^{-2\pi\alpha}]$ for $\alpha=1/3$, $3/8$, and
$1/2$
are shown in Fig.~\ref{fig:comp}.
\begin{figure}[htb]
\centerline{\psfig{file=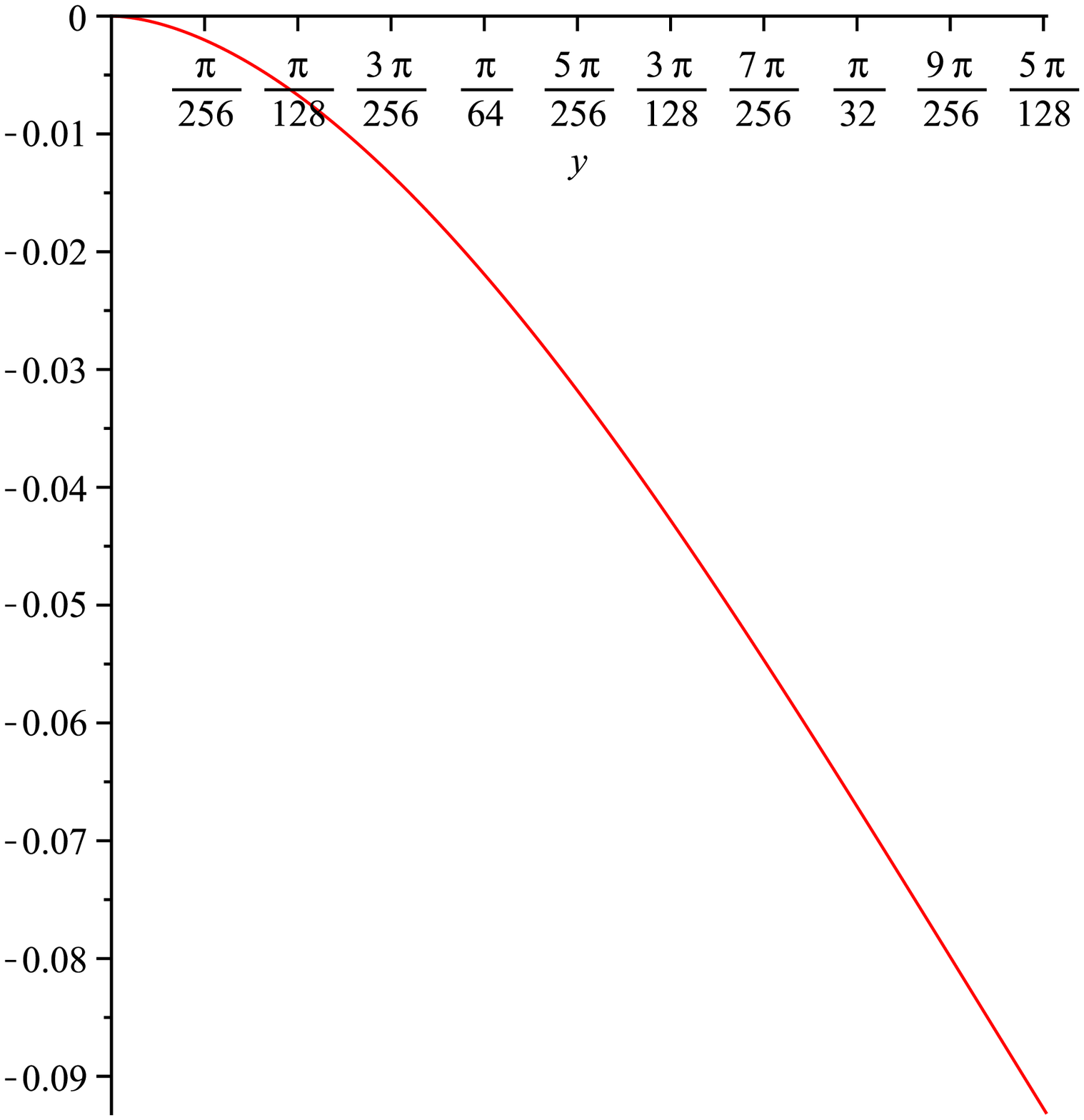,width=4cm,height=5cm,angle=0}
\psfig{file=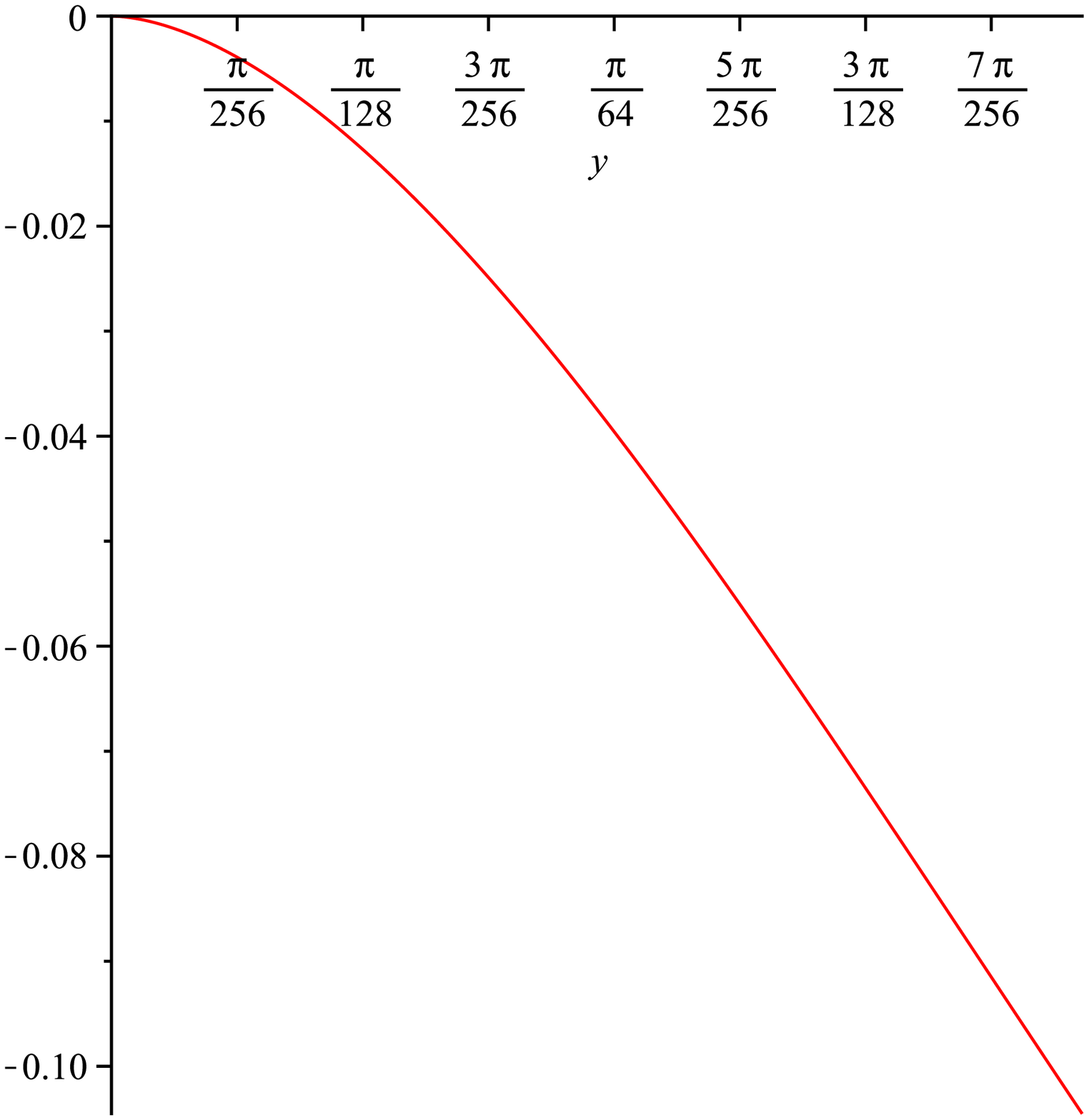,width=4cm,height=5cm,angle=0}
\psfig{file=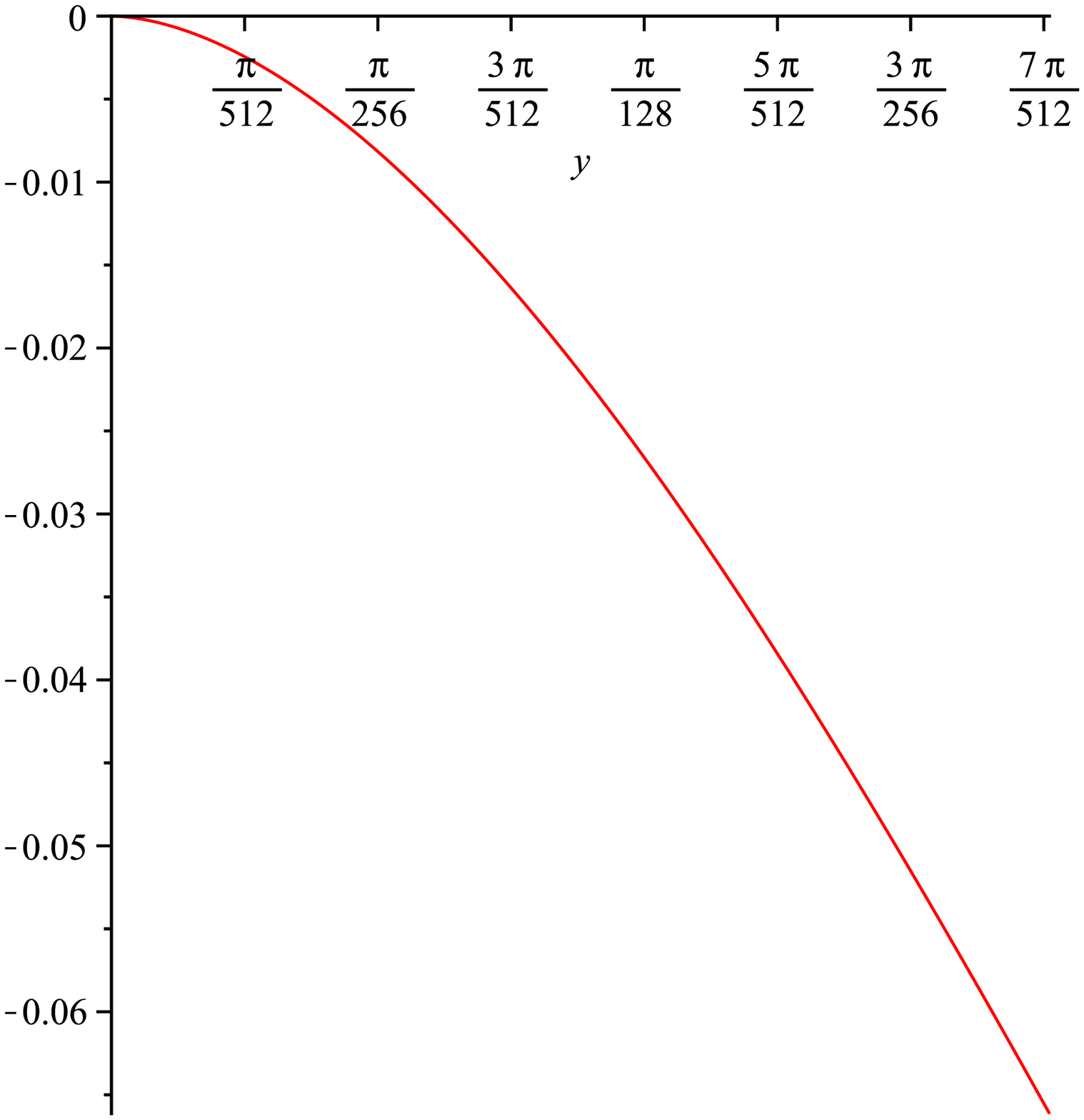,width=4cm,height=5cm,angle=0}
}
\caption{Graphs of
$\Phi(y)-(1+ay)$ on $y\in[0,e^{-2\pi\alpha}]$
\ for $\alpha=1/3$, $\alpha=3/8$, and
 $\alpha=1/2$; $a=2\cos(2\pi\alpha)$}
\label{fig:comp}
\end{figure}
This `proves' that inequality~\eqref{subs1} holds for all
$y\in[0,e^{-2\pi\alpha}]$ and we obtain, as a result, that
 the following theorem holds.
\begin{theorem}
\label{T:2}
For $\alpha=1/4$ and $\alpha=3/4$
$$
\|u\|^2_\infty\le\|A^{1/2}u\|\|u\|
\bigl(1-2e^{-4\pi\|A^{1/2}u\|/\|u\|}\bigr),
$$
while for $\alpha\in(1/4,3/4)$
$$
\|u\|^2_\infty\le\|A^{1/2}u\|\|u\|
\bigl(1+2\cos(2\pi\alpha)e^{-2\pi\|A^{1/2}u\|/\|u\|}\bigr).
$$
All constants are sharp.
\end{theorem}

\setcounter{equation}{0}
\section{Carlson--Landau inequalities}
\label{sec5}

One-dimensional inequalities  of $L_\infty$-$L_2$-$L_2$-type
with various boundary conditions are closely connected with
Carlson--Landau
inequalities and their various improvements.


\subsection*{Carlson--Landau inequality with
correction term}\label{ss1}

In the next theorem we
show that both inequality~\eqref{magn2} in the symmetric case
$\alpha=1/2$,  and inequality~\eqref{inter} are equivalent to \eqref{Landau},
while their refined forms \eqref{inter1} and  \eqref{magn3}
 are, in fact, equivalent and
provide  a sharp exponential correction term
 to Landau's improvement of Carlson's inequality.
 A sharp second-order Carlson-type inequality
 in the flavor of \eqref{Landau} is also given.
 The notation introduced in~\eqref{notation} is used in the
 following theorem.
\begin{theorem}\label{T:Landau}
The following inequality holds
\begin{equation}
\label{Landau1}
\biggl(\,\sum_{k=1}^\infty a_k\biggr)^2\le\pi\|a\|\|a\|_1
\left(1-2e^{-2\pi\|a\|_1/\|a\|}\right),
\end{equation}
where all constants are sharp and no extremals exist.

In the second-order case it holds
\begin{equation}\label{Landau2}
\biggl(\,\sum_{k=1}^\infty a_k\biggr)^2\le
\frac{\sqrt{2}\,\pi}{\sqrt[4]{27}}\coth(\pi/2)
\|a\|^{3/2}
\|a\|_2^{1/2}.
\end{equation}
Inequality~\eqref{Landau2} saturates at a unique
extremal
\begin{equation}\label{extremal4}
a_k=\frac1{(2k-1)^4+4}\,,
\end{equation}
\end{theorem}
\begin{proof}
Given a non-negative sequence $\{a_k\}_{k=1}^\infty$
we construct the sequence $\{b_k\}_{k=-\infty}^\infty$
by setting for $k=0,1,\dots$,
$$
b_0=b_{-1}:=a_1,\ b_1=b_{-2}:=a_2,\dots,
b_k=b_{-(k+1)}:=a_{k+1},\dots\,.
$$
Then for a periodic function
$$
u(x)=\sum_{k\in\mathbb{Z}}b_ke^{ikx}
$$
we have
$$
\|u\|_\infty=\sum_{k\in\mathbb{Z}}b_k=2\sum_{k=1}^\infty a_k,\quad
\|u\|^2=2\pi\sum_{k\in\mathbb{Z}}b_k^2=
4\pi\sum_{k=1}^\infty a_k^2,\\
$$
and
$$
\aligned
\|A^{1/2}u\|^2=2\pi\sum_{k=0}^\infty(k+1/2)^2b_k^2+
2\pi\sum_{k=-1}^{-\infty}(k+1/2)^2b_k^2=\\=
2\pi\sum_{k=1}^\infty(k-1/2)^2a_k^2+
2\pi\sum_{k=1}^{\infty}(k-1/2)^2b_{-k}^2=
4\pi\sum_{k=1}^\infty (k-1/2)^2a_k^2.
\endaligned
$$

Substituting this into the second inequality in \eqref{magn3}
gives~\eqref{Landau1}.

An alternative and a simpler proof was given in \S~\ref{sec1}
by using  \eqref{inter}
and its refinement~\eqref{inter1}.

As for the second-order inequality, for $u\in H^2(0,L)\cap H^1_0(0,L)$
we have the sharp inequality~\cite[Theorem 3.9]{IZ}
\begin{equation}\label{secondorder}
\|u\|_{\infty}^2\le\frac{\sqrt2}{\sqrt[4]{27}}\coth(\pi/2)
\|u\|^{3/2}\|u''\|^{1/2},
\end{equation}
saturating at a unique extremal function
(for $L=1$)
\begin{equation}\label{extremal}
u_*(x)=\sum_{k=1}^\infty\frac{(-1)^{k+1}\sin(2k-1)\pi x}
{(2k-1)^4+4}\,.
\end{equation}
Setting $L=1$, substituting $u(x)$ from~\eqref{u(x)}
into~\eqref{secondorder} and taking into account that $\|u''\|^2=16\pi^4
\sum_{k=1}^\infty(k-1/2)^4 a_k^4$ we obtain inequality
\eqref{Landau1}, while  the  unique extremal~\eqref{extremal4} is
produced by~\eqref{extremal}.

We finally observe that unlike all the previous Carlson-type inequalities
(namely, \eqref{Carlorig}, \eqref{Carl}, \eqref{Carl3},
\eqref{Landau}, \eqref{Landau1}) inequality ~\eqref{Landau2} has
a unique extremal~\eqref{extremal4}.
\end{proof}

\subsection*{Intermediate Carlson--Landau inequalities}\label{sss2}
 In conclusion we consider the  family
of intermediate  Carlson--Landau-type inequalities~\eqref{interCLalpha}
in the whole range $\alpha\in[0,1)$. In the case $\alpha=1/2$
the Carlson--Landau inequality was supplemented with an
exponentially small remainder term in Theorem~\ref{T:Landau}.

We now consider the region $\alpha\in[0,1/2)$. Obviously,
$k(\alpha)=\pi$ and we show below that there exists a (sharp) correction term:
\begin{equation}\label{interCL}
\left(\sum_{k=1}^\infty a_k\right)^2\le\pi\left(\sum_{k=1}^\infty a_k^2\right)^{1/2}\left(\sum_{k=1}^\infty
(k-\alpha)^2a_k^2\right)^{1/2}-
(1-2\alpha)\sum_{k=1}^\infty a_k^2.
\end{equation}
For $\alpha=0$ it is the classical Carlson inequality
supplemented with a lower order term in~\eqref{Carl}.

To prove~\eqref{interCL} we apply our method directly to sequences without going over to functions. We consider
the variational problem: for $D\ge (1-\alpha)^2$ find
\begin{equation}\label{VDseq}
\mathbb V(D,\alpha):=\sup\left\{\left(\sum_{k=1}^\infty a_k\right)^2\colon
\ \sum_{k=1}^\infty a_k^2=1, \
 \sum_{k=1}^\infty (k-\alpha)^2a_k^2=D\right\}.
\end{equation}
In complete analogy with~\eqref{freelambda} and
Theorem~\ref{T:Var} we find that
\begin{equation}\label{VDalpha}
\mathbb V(D,\alpha)=\min_{\alpha\ge-(1-\alpha)^2}
(\lambda+D)G(\lambda),
\end{equation}
where
\begin{equation}\label{Galpha}
G(\lambda)=\sum_{k=1}^\infty\frac1{(k-\alpha)^2+\lambda}.
\end{equation}
Using the Euler $\psi$-function $\psi(z)=\frac d{dz}\log \Gamma(z)$ and its representation
$$
\psi(z)=-\gamma+\sum_{n=0}^\infty\left(\frac1{n+1}-\frac1{n+z}\right)
$$
we factorize the denominator in~\eqref{Galpha} and find
\begin{equation}\label{Fal}
G(\lambda)=\frac{i(\psi(1-\alpha-i\sqrt{\lambda})-
\psi(1-\alpha+i\sqrt{\lambda}))}{2\sqrt{\lambda}}=:
\frac1{2\sqrt{\lambda}}F(\alpha,\lambda)\,.
\end{equation}
Using the Stirling expansion for the $\psi$-function
$$
\psi(z)=\ln z-\frac1{2z}-\frac1{12z^2}+O(z^{-3}),
$$
we get as $\lambda\to\infty$
$$
G(\lambda)=\frac{\pi}2\lambda^{-1/2}-\frac12(1-2\alpha)\lambda^{-1}+O(\lambda^{-2}).
$$
For the unique point of a minimum $\lambda(D)$ in~\eqref{VDalpha} we have the equation
$$
D=-\frac{G(\lambda)}{G'(\lambda)}-\lambda=
\lambda+\frac{2(1-2\alpha)}\pi\lambda^{1/2}+
\frac{4(1-2\alpha)^2}{\pi^2}+O(\lambda^{-1/2}),
$$
giving
\begin{equation}\label{expan-lam}
\lambda(D)=D-\frac{2(1-2\alpha)}\pi D^{1/2}-2(1-2\alpha)^2+O(D^{-1/2}).
\end{equation}
Substituting this into $\mathbb V(D,\alpha)=-\frac{G(\lambda(D))^2}{G'(\lambda(D))}$
we obtain the expansion
$$
 \mathbb{V}(D,a)=\pi D^{1/2}-(1-2a)-
 \frac{(2a-1)^2}{2\pi}D^{-1/2}+O(D^{-1}).
$$
The third term here is negative, hence
\begin{equation}\label{negative}
 \mathbb{V}(D,a)<\pi D^{1/2}-(1-2a)
\end{equation}
for all sufficiently large $D$. To see that this inequality holds for all $D$
we truncate the expansion~\eqref{expan-lam} by setting
$$
\lambda_*(D):=D-\frac{2(1-2\alpha)}\pi D^{1/2},
$$
and consider the explicitly given function
$$
 \mathbb{V}_*(D,\alpha):=(\lambda_*(D)+D)G(\lambda_*(D)).
 $$
Since by definition
$ \mathbb{V}(D,\alpha)\le\mathbb{V}_*(D,\alpha)$,  to establish~\eqref{negative} for all $D$
it suffices to show that the following function is negative
$$
R(D,\alpha):= \mathbb{V}_*(D,\alpha)-\pi D^{1/2}+(1-2a)
$$
for all $D\ge(1-\alpha)^2$.
We have the asymptotic expansion
$$
R(D,\alpha)=-(1-2\alpha)^2D^{-1/2}+O(D^{-1})
$$
giving that $R(D,\alpha)<0$ for all sufficiently large $D$.
The graphs of $R(D,\alpha)$ for different $\alpha$ are shown
in Fig.~\ref{fig:comp2}, where one can see a very rapid convergence to $0$
for $\alpha=1/2$.
\begin{figure}[htb]
\centerline{\psfig{file=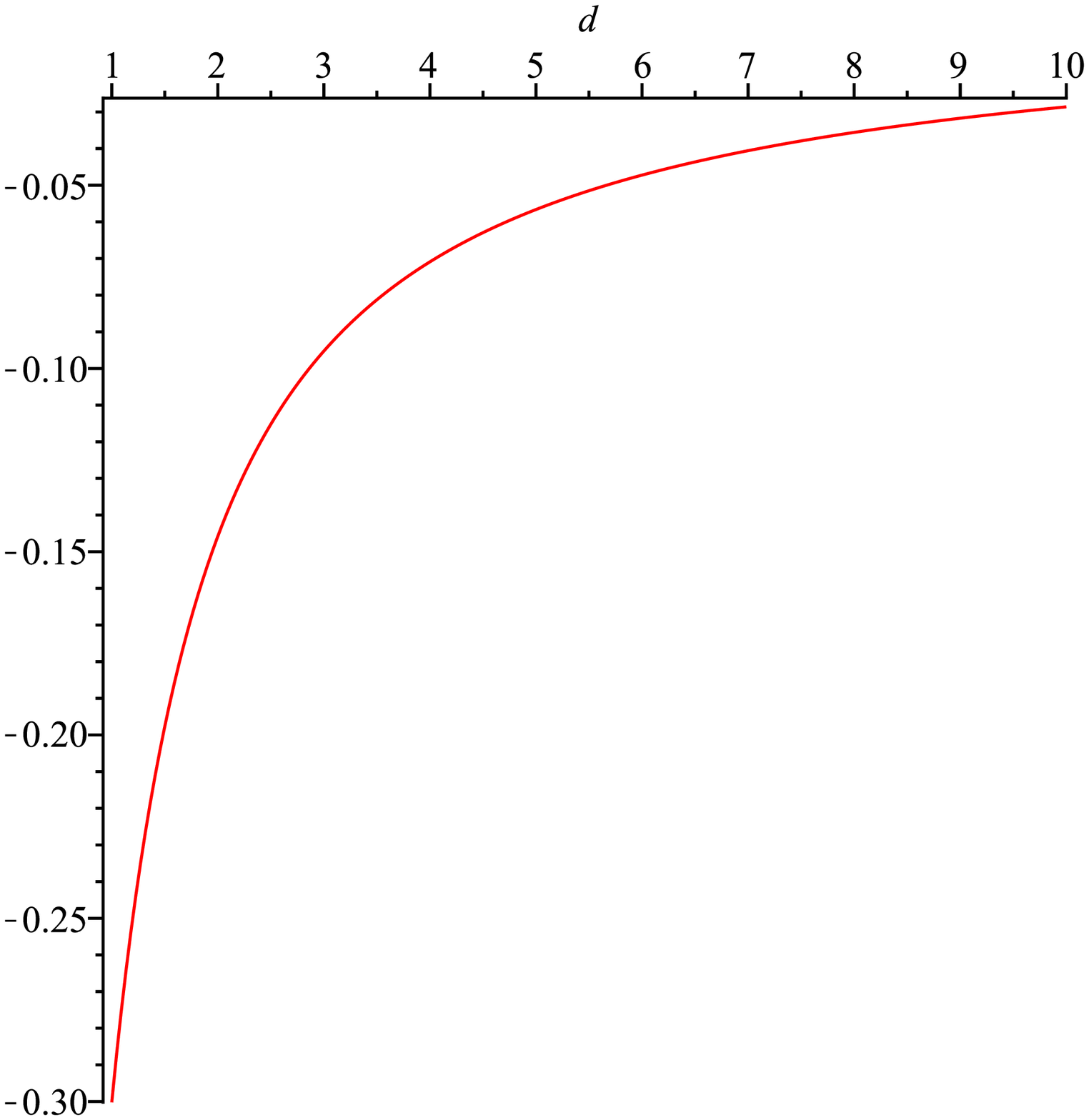,width=3cm,height=5cm,angle=0} 
\psfig{file=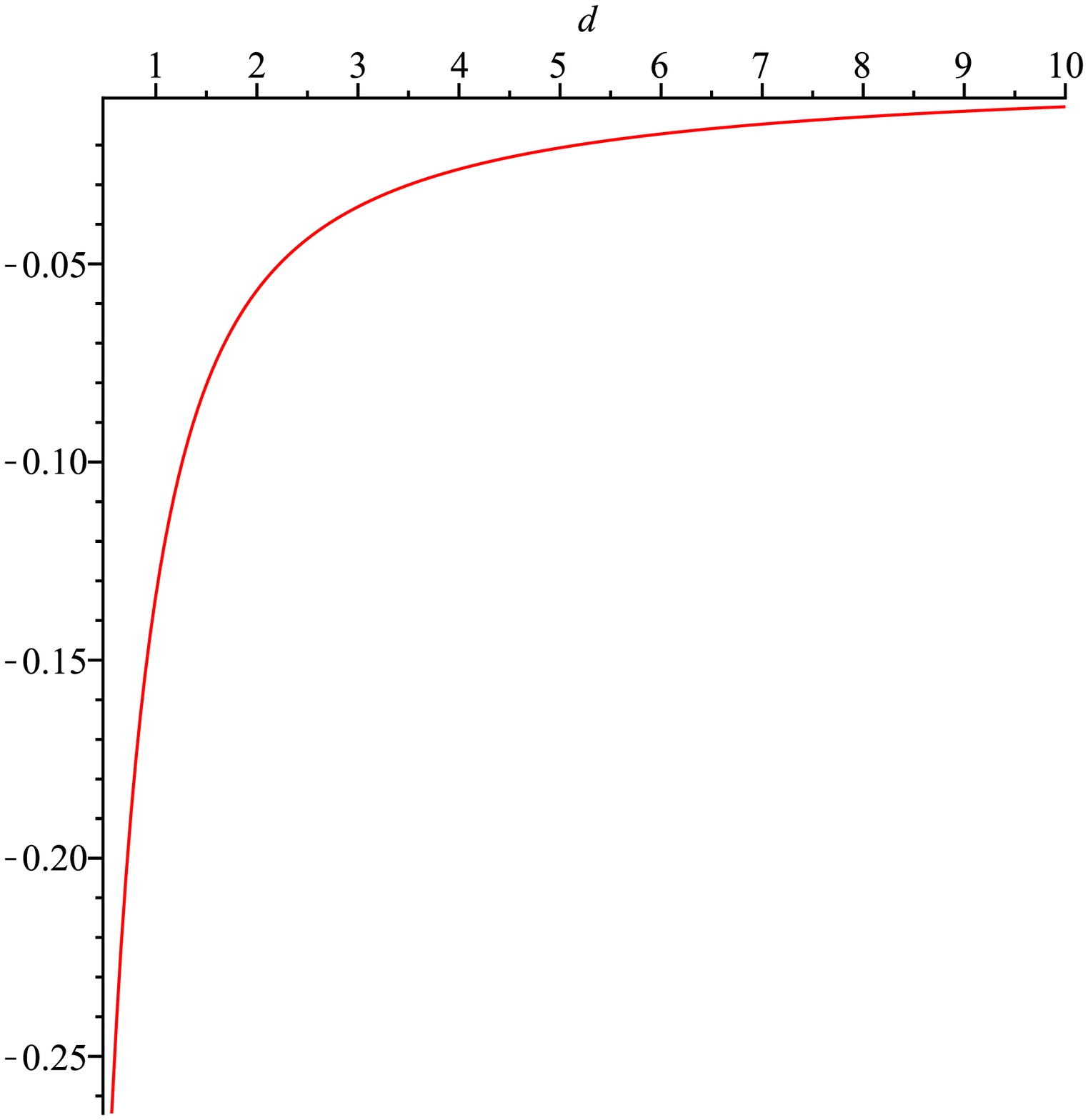,width=3cm,height=5cm,angle=0}
\psfig{file=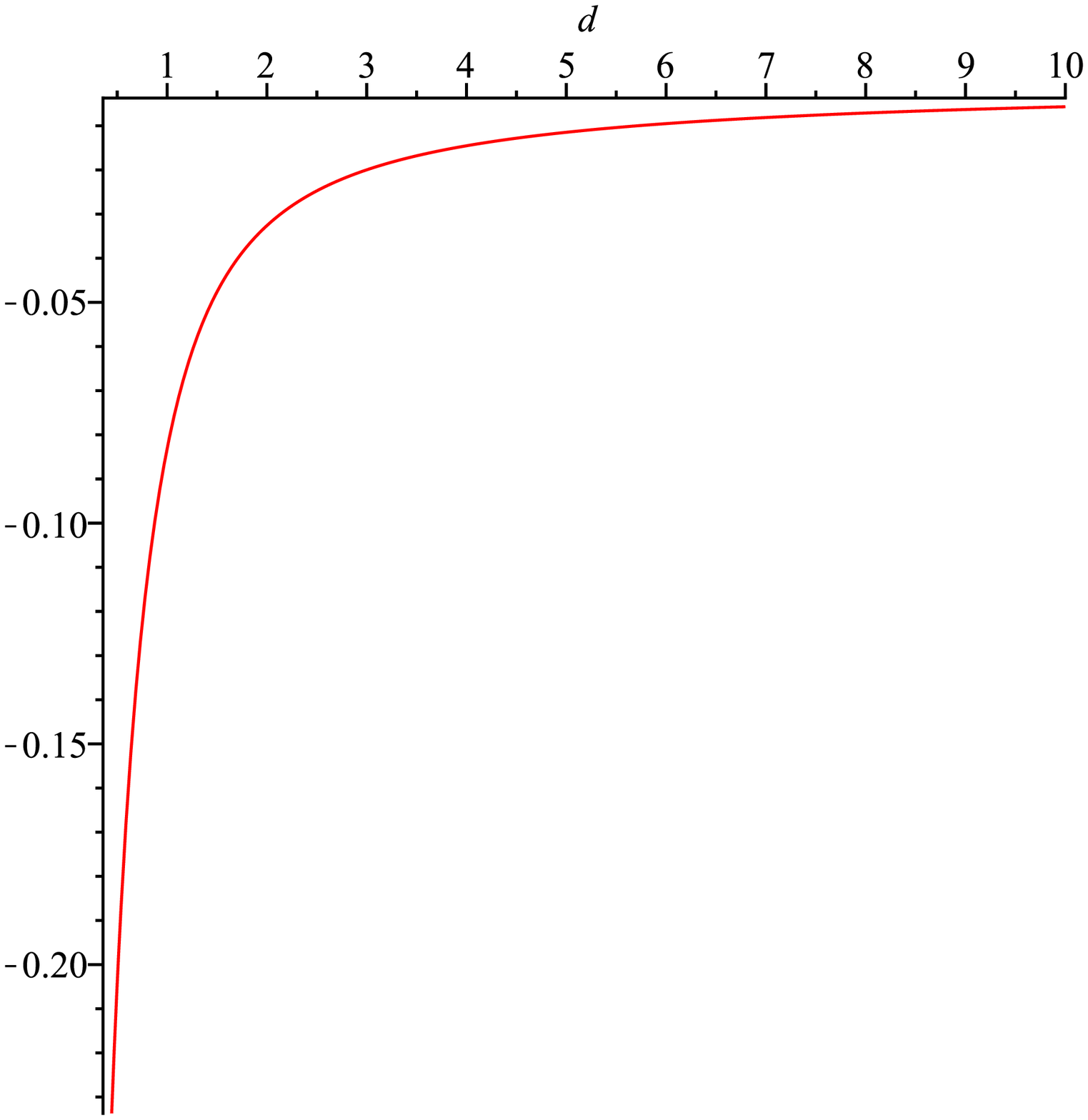,width=3cm,height=5cm,angle=0}
\psfig{file=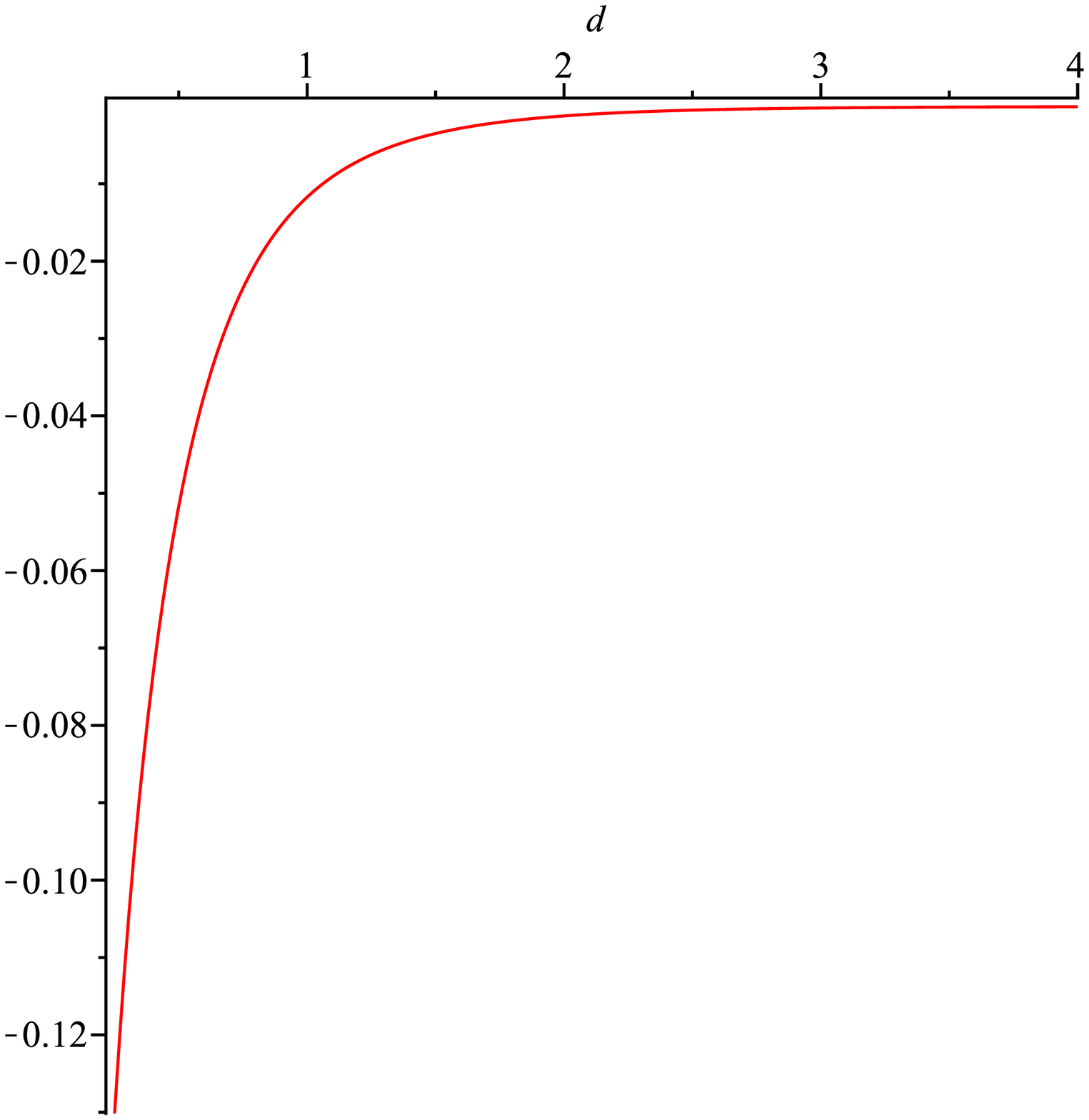,width=3cm,height=5cm,angle=0}
}
\caption{Graphs of
$R(D,\alpha)$
\ for $\alpha=0$, $\alpha=1/4$,
$\alpha=1/3$, and $\alpha=1/2$.}
\label{fig:comp2}
\end{figure}

The case when $\alpha\in(1/2,1)$ is qualitatively different and
very similar to the `magnetic' inequalities in
 Theorem~\ref{Th:magnetic}.
Namely, $k(\alpha)>\pi$ and there exists a (unique) extremal
in~\eqref{interCLalpha}. In fact, repeating word for word the argument in
Remark~\ref{R:direct} (replacing $\sum_{k\in\mathbb{Z}}$ by
$\sum_{k\in\mathbb{N}}$) we obtain that
$$
k(\alpha)=2\sup_{\lambda>0}\sqrt{\lambda}G(\lambda)=
\sup_{\lambda>0}F(\alpha,\lambda),
$$
where $F$ is defined in~\eqref{Fal}.
Since $\lim_{\lambda\to\infty}F(\alpha,\lambda)=\pi$,
it follows that $k(\alpha)\ge\pi$. The supremum is, in fact,
a maximum, that is attained at a (unique) point
$\lambda_*(\alpha)$, for which $\lambda_*(\alpha)\sim (1-\alpha)^2$
and $k(\alpha)\sim 1/(1-\alpha)$ as $\alpha\to 1^-$
(this easily follows from the asymptotic behavior of
$\psi(z)$ near $0$: $\psi(z)=-\gamma-1/z+O(z^2)$), see
Fig.~\ref{fig:comp3}.
\begin{figure}[htb]
\centerline{\psfig{file=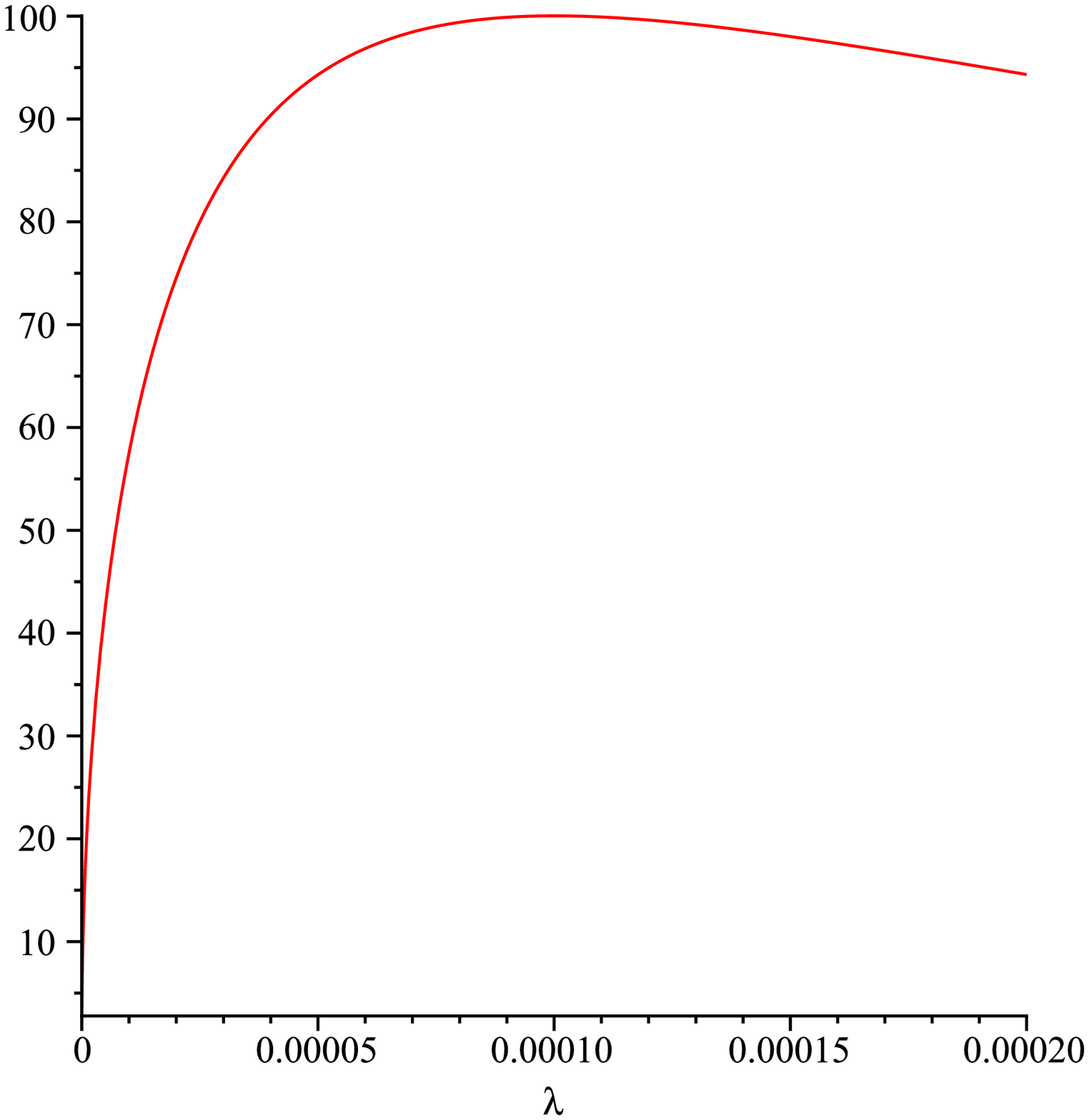,width=4.2cm,height=5cm,angle=0}
\psfig{file=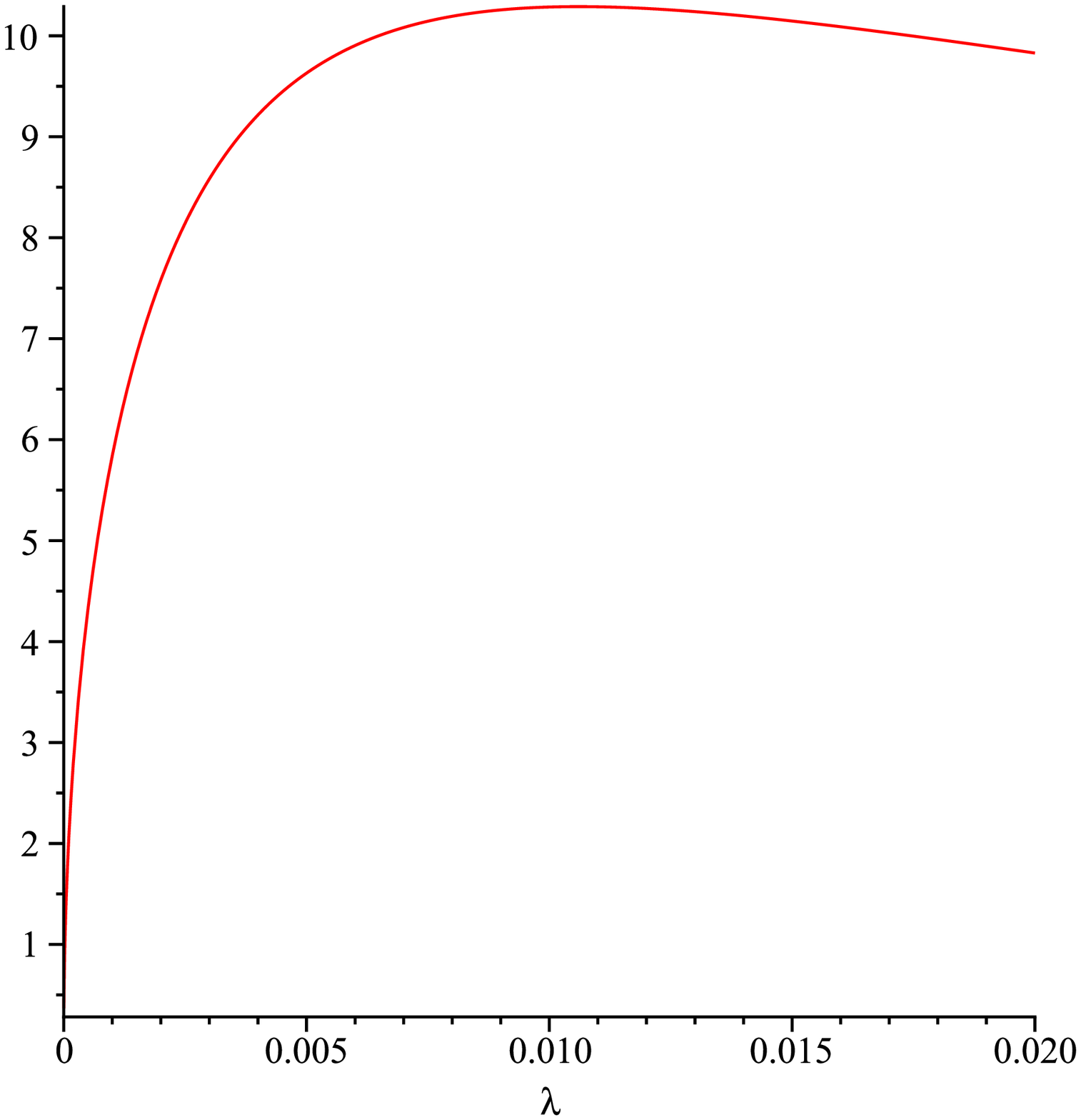,width=4.2cm,height=5cm,angle=0}
\psfig{file=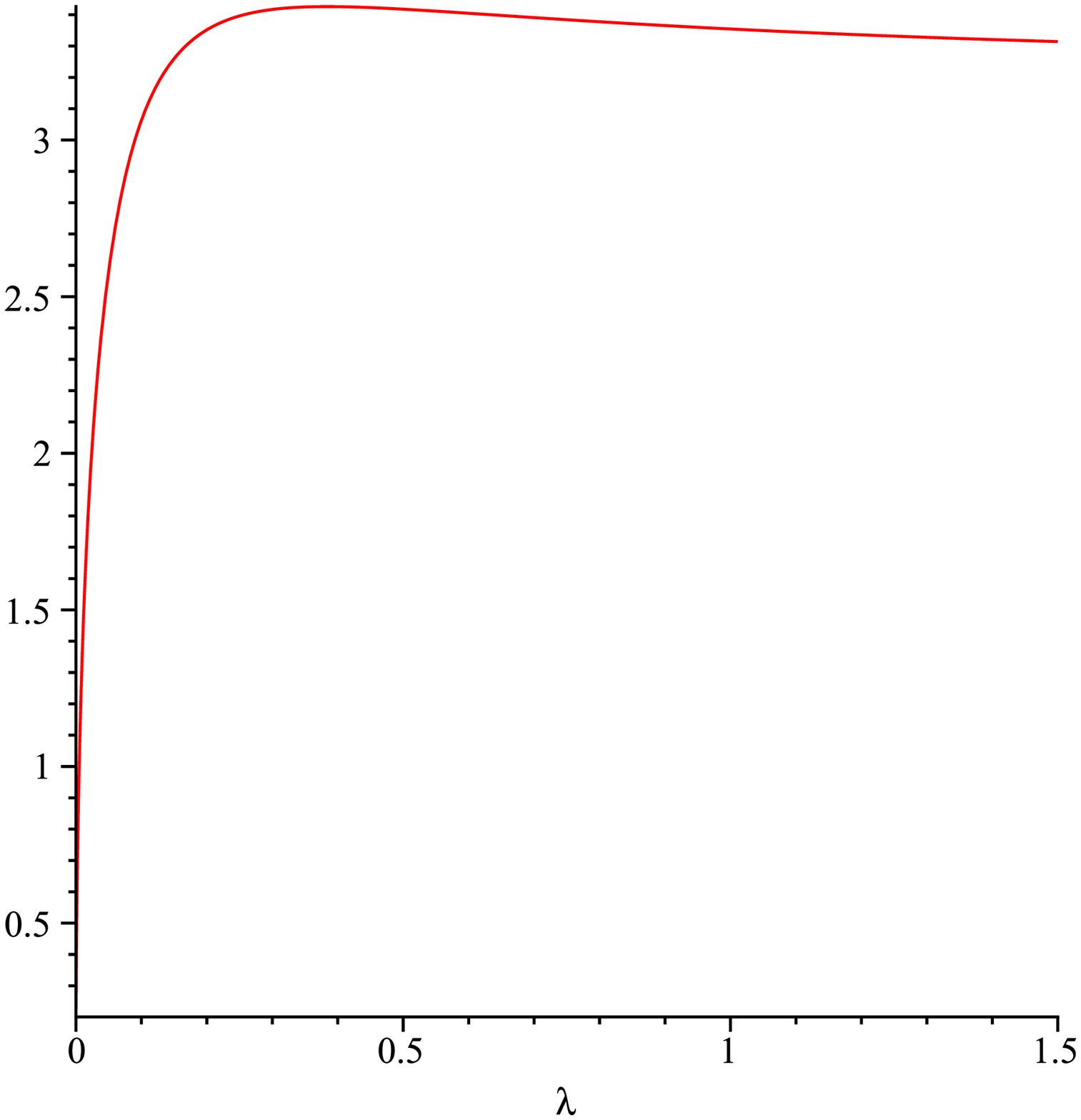,width=4.2cm,height=5cm,angle=0}
}
\caption{Global maximums of $F(\alpha, \lambda)$
 for $\alpha=0.99$, $\alpha=0.9$,
$\alpha=0.6$.}
\label{fig:comp3}
\end{figure}

Thus, with the help of reliable computer calculations we obtain the following result.
\begin{theorem}\label{Th:Intermediate}
Inequality \eqref{interCL} holds for
$\alpha\in[0,1/2)$. The constants are sharp,
no extremals exist.

For $\alpha\in(1/2,1)$ the sharp constant in~\eqref{interCLalpha}
is
$$
k(\alpha)=\max_{\lambda>0}
i\bigl(\psi(1-\alpha-i\sqrt{\lambda})-
\psi(1-\alpha+i\sqrt{\lambda})\bigr).
$$
The maximum is attained at a (unique) point $\lambda_*(\alpha)$
and there exists a unique extremal
$$
a_k=\frac1{(k-\alpha)^2+\lambda_*(\alpha)}.
$$
\end{theorem}

\setcounter{equation}{0}
\section{Lieb--Thirring estimates for magnetic
Schr\"odinger operators}\label{sec6}

\subsection*{One-dimensional
 Sobolev inequalities for matrices}\label{ss61}

In this section we give an alternative proof of the main result in
\cite{D-L-L} along with its generalization to higher order
derivatives and magnetic operators.

Let $\{\phi_n\}_{n=1}^N$ be an orthonormal family of
vector-functions
$$
\phi_n(x)=(\phi_n(x,1),\dots,\phi_n(x,M))^T
$$
and
$$
(\phi_n,\phi_m)
=
\sum_{j=1}^M\int_D\phi_n(x,j)\overline{\phi_m(x,j)}dx
=\int_D \phi_n(x)^T\overline{\phi_m(x)}dx=\delta_{nm}.
$$
Here $D=\mathbb{R}$ or $D=\mathbb{S}^1$.  In the latter case we
assume that for all $n$ and $j$
$$
\int_0^{2\pi}\phi_n(x,j)dx=0.
$$

We consider the $M\times M$ matrix $U(x,y)$
\begin{equation}\label{matrixU}
U(x,y)=\sum_{n=1}^N\phi_n(x)\overline{\phi_n(y)}^T
\end{equation}
so that
$[U(x,y)]_{jk}=\sum_{n=1}^N\phi_n(x,j)\overline{\phi_n(y,k)}$.
Clearly,
$$
U(x,y)^*=U(y,x)
$$
and by orthonormality
$$
\aligned
\int_D U(x,y)U(y,z)dy=\sum_{n,n'=1}^N\int_D \phi_n(x)
\overline{\phi_n(y)}^T\phi_{n'}(y)\overline{\phi_{n'}(z)}^Tdy=\\=
\sum_{n=1}^N \phi_n(x)
\overline{\phi_{n}(z)}^T=U(x,z).
\endaligned
$$
In addition, $U(x,x)$ is positive semi-definite, since $(U(x,x)a,a)=
\sum_{n=1}^N|a^T\phi_n(x)|^2\ge0$.

\begin{theorem}\label{T:matrix}
Let $m>1/2$. Then
\begin{equation}\label{trace}
\int_D\Tr[U(x,x)^{2m+1}]dx\le C(m)^{2m}\sum_{n=1}^N\sum_{j=1}^M
\int_D|\phi^{(m)}_n(x,j)|^2dx,
\end{equation}
where $C(m)$ is defined in~\eqref{Sob}. In particular, for $m=1,2$
$$
\aligned
\int_D\Tr[U(x,x)^{3}]dx\le\sum_{n=1}^N\sum_{j=1}^M
\int_D|\phi'_n(x,j)|^2dx,\\
\int_D\Tr[U(x,x)^{5}]dx\le \frac 4{27}\sum_{n=1}^N\sum_{j=1}^M
\int_D|\phi''_n(x,j)|^2dx.
\endaligned
$$
\end{theorem}
\begin{proof}
We first consider the periodic case.
We write
$$
\tilde U(n,x)  = \int_0^{2\pi}  \frac{ e^{- i y n}}{\sqrt{2\pi}} U(y,x) dy
$$
so that
$$
U(y,x) = \sum_{k\in\mathbb{Z}_0}  \frac{e^{ i y k}}{\sqrt{2\pi}}  \tilde U(k,x)  \ .
$$
We have
\begin{equation}\label{orthint}
\sum_{k\in\mathbb{Z}_0} \tilde U(k,x)^* \tilde U(k,x) =
\int_0^{2\pi}  U(y,x)^* U(y,x) dy  = U(x,x),
\end{equation}
where $\mathbb{Z}_0=\mathbb{Z}\setminus \{0\}$, and
we further have
$$
\sum_{k\in\mathbb{Z}_0} |k|^{2m}
\tilde U(k,x)^* \tilde U(k,x) =
\int_0^{2\pi} [\partial_y^{(m)} U (y,x)]^* \partial_y^{(m)} U(y,x)  dy
$$
so that by orthonormality
\begin{equation}\label{Dtrace}
\aligned
\Tr\left[\int_0^{2\pi} \sum_{k\in \mathbb{Z}} |k|^{2m}
\tilde U(k,x)^* \tilde U(k,x)  dx \right]=\\=
\Tr\left[\int_0^{2\pi} \int_0^{2\pi} \sum_{n,n'=1}^{N}
\phi_{n'}^{(m)}(y)\overline{\phi_{n'}(x)}^T
\phi_n(x)\overline{\phi_n^{(m)}(y)}^T
 dx dy\right]=\\
\Tr\left[ \int_0^{2\pi} \sum_{n=1}^{N}
\phi_{n}^{(m)}(y)\overline{\phi_n^{(m)}(y)}^T dy\right]=
\sum_{n=1}^N\sum_{j=1}^M
\int_0^{2\pi}|\phi^{(m)}_n(x,j)|^2dx.
\endaligned
\end{equation}

Now consider
$$
\aligned
\Tr[U(x,x)^{2m+1}] = \sum_{k\in\mathbb{Z}_0}
\Tr[U(x,x)^{2m} \tilde U(k,x) ]  \frac{e^{ i x k}}{\sqrt{2 \pi}}=\\
\sum_{k\in\mathbb{Z}_0}
\Tr\biggl [ [|k|^{2m} I + \Lambda(x)^{2m}]^{-1/2} U(x,x)^{2m}
\tilde U(k,x)\times\\ [|k|^{2m} I + \Lambda(x)^{2m}]^{1/2}\biggr ]
\frac{e^{ i x k}}{\sqrt{2 \pi}}\,,
\endaligned
$$
where $\Lambda(x)$ is an arbitrary positive definite matrix. Using
 below the  Cauchy--Schwarz inequality for
matrices  we get the upper bounds
$$
\aligned
\Tr[U(x,x)^{2m+1}] \le\\
\frac 1{\sqrt{2 \pi}} \sum_{k\in\mathbb{Z}_0}
\left|\Tr\left [ [|k|^{2m} I + \Lambda(x)^{2m}]^{-1/2} U(x,x)^{2m}\right.\right.\times\\
\left.\left.\tilde U(k,x)[|k|^{2m} I + \Lambda(x)^{2m}]^{1/2}\right ]
\right|\le\\
\frac 1{\sqrt{2 \pi}} \sum_{k\in\mathbb{Z}_0}
\left(\Tr\left [ U(x,x)^{2m}[|k|^{2m} I +\Lambda(x)^{2m}]^{-1}
U(x,x)^{2m}\right]\right)^{1/2}\times\\
\left(\Tr\left[
[|k|^{2m} I + \Lambda(x)^{2m}]
\tilde U(k,x)^*\tilde U(k,x)\right]\right)^{1/2}\le\\
\frac 1{\sqrt{2 \pi}} \left(\sum_{k\in\mathbb{Z}_0}
\Tr\left [ U(x,x)^{2m}[|k|^{2m} I + \Lambda(x)^{2m}]^{-1}
U(x,x)^{2m}\right]\right)^{1/2}\times\\
\left(\sum_{k\in\mathbb{Z}_0}\Tr\left[
[|k|^{2m} I + \Lambda(x)^{2m}]
\tilde U(k,x)^*\tilde U(k,x)\right]\right)^{1/2}\,.
\endaligned
$$
For the first factor we have
$$
\aligned
\sum_{k\in\mathbb{Z}_0}
\Tr\left [ U(x,x)^{2m}[|k|^{2m} I + \Lambda(x)^{2m}]^{-1}
U(x,x)^{2m}\right]=\\
\Tr\left [ U(x,x)^{2m}\sum_{k\in\mathbb{Z}_0}[|k|^{2m} I + \Lambda(x)^{2m}]^{-1}
U(x,x)^{2m}\right]<\\
2\pi c(m)\Tr\left [ U(x,x)^{2m}\Lambda(x)^{-(2m-1)}
U(x,x)^{2m}\right],
\endaligned
$$
where we have used the  matrix inequality
\begin{equation}\label{matrineq}
\sum_{k\in\mathbb{Z}_0}[|k|^{2m} I + \Lambda(x)^{2m}]^{-1}<
2\pi c_0(m)\Lambda(x)^{-(2m-1)},\quad c_0(m)=\frac1{2m\sin\frac{\pi}{2m}}.
\end{equation}
In fact, the action of the matrix on the left-hand side on each
eigenvector $e=e(x)$ of $\Lambda(x)$ with eigenvalue
$\lambda=\lambda(x)>0$ from the orthonormal basis
$\{e_j(x),\ \lambda_j(x)\}_{j=1}^M$  results in multiplication of it by the number
$\sum_{k\in\mathbb{Z}_0}\frac1{|k|^{2m}+\lambda^{2m}}$ for which we
have
$$
\aligned
\sum_{k\in\mathbb{Z}_0}\frac1{|k|^{2m}+\lambda^{2m}}=
\lambda^{-(2m-1)}\frac1\lambda\sum_{k\in\mathbb{Z}_0}\frac1{(|k|/\lambda)^{2m}+1}<\\
\lambda^{-(2m-1)}\,2\int_{0}^\infty\frac{dx}{x^{2m}+1}=
\lambda^{-(2m-1)}2\pi c_0(m),
\endaligned
$$
since the function $1/({x^{2m}+1})$ is monotone decreasing on
$[0,\infty)$.

For the second factor we simply have
$$
\aligned
\sum_{k\in\mathbb{Z}_0}\Tr\left[
[|k|^{2m} I + \Lambda(x)^{2m}]
\tilde U(k,x)^*\tilde U(k,x)\right]=\\
\sum_{k\in\mathbb{Z}_0}\Tr
\left[|k|^{2m} \tilde U(k,x)^*\tilde U(k,x)\right] +
\sum_{k\in\mathbb{Z}_0}\Tr\left[\Lambda(x)^{2m}
\tilde U(k,x)^*\tilde U(k,x)\right].
\endaligned
$$

If we now chose $\Lambda(x)=\beta(U(x,x)+\varepsilon I)$ and let $\varepsilon\to0$
we obtain (observing that $\lambda^{4m}/(\lambda+\varepsilon)^{2m-1}\to
\lambda^{2m+1}$ as $\varepsilon\to0$ for $\lambda\ge0$; this is required in case when $U(x,x)$
is not invertible)
$$
\aligned
\Tr[U(x,x)^{2m+1}] \le
c_0(m)^{1/2}\beta^{-(2m-1)/2}\Tr[U(x,x)^{2m+1}]^{1/2}\times\\
\left(
\sum_{k\in\mathbb{Z}_0}\Tr
\left[|k|^{2m} \tilde U(k,x)^*\tilde U(k,x)\right]
+\beta^{2m}\Tr[U(x,x)^{2m+1}]
\right),
\endaligned
$$
where we have also used~\eqref{orthint}, or
$$
\aligned
&\Tr[U(x,x)^{2m+1}] \le\\
c_0(m)
\biggl(\beta^{-(2m-1)}
\sum_{k\in\mathbb{Z}_0}
&\Tr
\left[|k|^{2m} \tilde U(k,x)^*
\tilde U(k,x)\right]
+\beta\Tr[U(x,x)^{2m+1}]
\biggr).
\endaligned
$$
If we optimize over $\beta$, we obtain
$$
\aligned
\Tr[U(x,x)^{2m+1}] \le c_0(m)\frac{1}{\theta^\theta(1-\theta)^{1-\theta}}\times\\
\biggl(\Tr[U(x,x)^{2m+1}]\biggr)^{\theta}
\biggl(
\sum_{k\in\mathbb{Z}_0}
\Tr
\left[|k|^{2m} \tilde U(k,x)^*
\tilde U(k,x)\right]
\biggr)^{1-\theta}
\endaligned
$$
or
$$
\Tr[U(x,x)^{2m+1}] \le
C(m)^{2m}
\sum_{k\in\mathbb{Z}_0}
\Tr
\left[|k|^{2m} \tilde U(k,x)^*
\tilde U(k,x)\right].
$$
If we integrate with respect to $x$ and use~\eqref{Dtrace},
 we obtain~\eqref{trace}.

In the case of $x\in\mathbb{R}$ the proof is similar. We
use the Fourier transform instead of the Fourier series
and the matrix equality
$$
\int_{-\infty}^\infty[|p|^{2m} I + \Lambda(x)^{2m}]^{-1}dp=
2\pi c_0(m)\Lambda(x)^{-(2m-1)}
$$
instead of~\eqref{matrineq}.
\end{proof}

The one-dimensional periodic magnetic case is treated similarly.
Suppose that as before we have a family of orthonormal
periodic vector-functions (no zero average condition is assumed).
As before we construct the matrix $U$~\eqref{matrixU}.
\begin{theorem}\label{T:matrix_magn}
The following inequality holds
\begin{equation}\label{trace_mag}
\int_0^{2\pi}\Tr[U(x,x)^{3}]dx\le K(\alpha)^2\sum_{n=1}^N
\sum_{j=1}^M  \int_0^{2\pi}|(i \partial_x -a(x)) \phi_n(x,j)|^2 dx,
\end{equation}
where $K(\alpha)$ is defined in~\eqref{Kalpha1}.
\end{theorem}
\begin{proof} We define the matrix Fourier coefficients for all
$n\in\mathbb{Z}$.
We now  have
$$
\sum_{k\in\mathbb{Z}}  \tilde U(k,x)^* \tilde U(k,x) =
\int_0^{2\pi}  U(y,x)^* U(y,x) dy  = U(x,x) \ ,
$$
and
$$
\sum_{k\in\mathbb{Z}} (k+\alpha)^2   \tilde U(k,x)^* \tilde U(k,x)
 = \int_0^{2\pi} [(i \partial_y -a)U (y,x)]^*  (i \partial_y - a) U(y,x)  dy
$$
so that instead of \eqref{Dtrace} we now have
\begin{equation}\label{magnD}
\int_0^{2\pi} \sum_{k\in\mathbb{Z}} (k+\alpha)^2 \tilde U(k,x)^* \tilde U(k,x)  dx =
\sum_{n=1}^N \sum_{j=1}^M  \int_0^{2\pi} |(i \partial_x -a) \phi_n(x,j)|^2 dx.
\end{equation}

As in the proof of  Theorem~\ref{T:matrix}  we have
$$
\aligned
\Tr[U(x,x)^{3}]\le  \qquad\qquad\\
\frac 1{\sqrt{2 \pi}} \left(\sum_{k\in\mathbb{Z}}
\Tr\left [ U(x,x)^{2}[(k+\alpha)^{2} I + \Lambda(x)^{2}]^{-1}
U(x,x)^{2}\right]\right)^{1/2}\times\\
\left(\sum_{k\in\mathbb{Z}}\Tr\left[
[(k+\alpha)^{2} I + \Lambda(x)^{2}]
\tilde U(k,x)^*\tilde U(k,x)\right]\right)^{1/2}\,.
\endaligned
$$
Now as a matrix inequality
$$
\sum_{k\in\mathbb{Z}}
[(k+\alpha)^{2} I + \Lambda(x)^{2}]^{-1}<\pi K(\alpha)\Lambda(x)^{-1},
$$
since the action of the matrix on the left-hand side on an
eigenvector $e=e(x)$ of $\Lambda(x)$ with eigenvalue
$\lambda=\lambda(x)$ is a multiplication of it by the number
$\sum_{k\in\mathbb{Z}}\frac1{(k+\alpha)^{2}+\lambda^{2}}$
and in view of~\eqref{Kalpha}
$$
\sum_{k\in\mathbb{Z}}\frac1{(k+\alpha)^{2}+\lambda^{2}}
<\pi K(\alpha)\frac1\lambda\,.
$$
If we again set $\Lambda(x)=\beta(U(x,x)+\varepsilon I)$
and let $\varepsilon\to0$ we obtain
\begin{multline*}
\Tr[U(x,x)^3] \le \\ \le \frac{K(\alpha) }{2 }  \biggl(\beta^{-1}
\sum_{k\in\mathbb{Z}}  (k+\alpha)^2\Tr[\tilde U(k,x)^*\tilde U(k,x)]+
\beta \Tr[U(x,x)^3]\biggr)\,.
\end{multline*}
If we optimize over $\beta$, we get
\begin{multline*}
\Tr[U(x,x)^3]  \le \\\le K(\alpha)
\biggl(\sum_{k\in\mathbb{Z}} (k+\alpha)^2  \Tr[\tilde U(k,x)^*  \tilde U(k,x)]\biggr)^{1/2}
\biggl( Tr[U(x,x)^3]\biggr)^{1/2}
\end{multline*}
and hence
$$
\Tr[U(x,x)^3]  \le K(\alpha)^2 \sum_{k\in\mathbb{Z}}  (k+\alpha)^2 \Tr[\tilde U(k,x)^*
  \tilde U(k,x)]
$$
from which our inequality follows by integration in  $x$ and
using~\eqref{magnD}.
\end{proof}
\begin{remark}\label{R:E-F}
{\rm
In the scalar case $M=1$ inequality \eqref{trace_mag} becomes
\begin{equation}\label{EF}
\int_0^{2\pi}\biggl(\sum_{n=1}^N|\phi_n(x)|^2\biggr)^3dx\le
K(\alpha)^2\sum_{n=1}^N \int_0^{2\pi}
\left|\left(i\partial_x-a(x)\right) \phi_n(x)\right|^2dx
\end{equation}
and follows from \eqref{magnetic}
by the method  of~\cite{E-F}.
}
\end{remark}

Theorem~\ref{T:matrix_magn} is equivalent to the estimate of
the negative trace
of the magnetic Schr\"odinger  operator
\begin{equation}\label{matr_oper}
H=\left(i\frac d{dx}-a(x)\right)^2-V
\end{equation}
in $L_2(\mathbb{S}^1)$ with matrix-valued potential $V$.
\begin{theorem}\label{Th:matrS1}
Let $V\ge0$ be a $M\times M$ Hermitian matrix such that $\Tr V^{3/2} \in L_1(\mathbb R)$.
Then the spectrum of operator~\eqref{matr_oper} is discrete and the negative eigenvalues $-\lambda_n\le0$
satisfy the estimate
\begin{equation}\label{matrtrace}
\sum_n\lambda_n\le\frac 2{3\sqrt{3}}K(\alpha)
\int_0^{2\pi}  \Tr[V(x)^{3/2}]dx.
\end{equation}
\end{theorem}
\begin{proof} (See~\cite{D-L-L}.) Let $\{\phi_n\}_{n=1}^N$
be the orthonormal eigen-vector functions corresponding
to $\{-\lambda_n\}_{n=1}^N$:
$$
\left(i\frac d{dx}-a(x)\right)^2\phi_n-V\phi_n=-\lambda_n\phi_n.
$$
Then, using~\eqref{trace_mag} and H\"older's inequality for
 traces
 $$\Tr[AB]\le(\Tr([(A^*A)^{p/2}])^{1/p}(\Tr([(B^*B)^{p'/2}])^{1/p'}$$
 and setting below $X:=\int_0^{2\pi}\Tr [U(x,x)^3]dx$
we obtain
$$
\aligned
&\sum_{n=1}^N\lambda_n=\\
&-\sum_{n=1}^N \sum_{j=1}^M  \int_0^{2\pi} |(i \partial_x -a(x)) \phi_n(x,j)|^2 dx
+\int_0^{2\pi}\Tr [V(x)U(x,x)]dx\le\\
&\le
\left(\int_0^{2\pi}\Tr[V(x)^{3/2}]dx\right)^{2/3}
X^{1/3}-K(\alpha)^{-2}X.
\endaligned
$$
Calculating the maximum with respect to $X$ we
obtain~\eqref{matrtrace}.
\end{proof}

Let %
\begin{equation}\label{class}
L_{\gamma,d}^{\mathrm{cl}} = \frac{1}{(2\pi)^d}\, \int_{\mathbb R^d}  (1-|\xi|^2)_+^\gamma\, d\xi=
\frac{\Gamma(\gamma+1)}{2^d\pi^{d/2}\Gamma(\gamma+d/2+1)}.
\end{equation}
By using the Aizenmann-Lieb argument \cite{AisLieb} we 
immediately obtain the following statement for the Riesz means 
of the eigenvalues for magnetic Schr\"odinger operators with matrix-valued potentials.
\begin{corollary}\label{1D gamma-moments}
Let $V\ge0$ be a  $M\times M$ Hermitian matrix, such that 
$\Tr V^{\gamma + 1/2} \in L_1(0,2\pi)$.
Then for any $\gamma \ge 1$  the negative eigenvalues of 
the operator \eqref{matr_oper} satisfy the inequalities
\begin{equation*}
\sum \lambda_n^\gamma \le L_{\gamma,1} \int_0^{2\pi} \Tr[V(x)^{1/2 + \gamma}] \, dx,
\end{equation*}
where
\begin{equation*}
L_{\gamma,1} \le \frac 2{3\sqrt{3}}K(\alpha) \frac{L_{\gamma,1}^{\mathrm{cl}}}{L_{1,1}^{\mathrm{cl}}}=
\frac\pi{\sqrt{3}}K(\alpha)L_{\gamma,1}^{\mathrm{cl}}.
\end{equation*}
\end{corollary}
\begin{proof}It is enough to prove this result for smooth matrix-valued potentials.
Note that Theorem \ref{Th:matrS1} is equivalent to
$$
\sum_n\lambda_n\le\frac 2{3\sqrt{3}}K(\alpha) 
(L_{1,1}^{\mathrm{cl}})^{-1}  \,\int_0^{2\pi}
\int_{-\infty}^\infty \Tr\left[\left(|\xi|^2 - V(x)\right)_-\right]\, \frac{d\xi dx}{2\pi}.
$$
Scaling gives the simple identity for all $s\in \mathbb R$
$$
s_-^\gamma = C_\gamma \, \int_0^\infty t^{\gamma-2} (s+t)_- dt ,
 \qquad C_\gamma^{-1} = \mathcal B(\gamma - 1, 2),
$$
where $\mathcal B$ is the Beta function. Let
$\{\mu_j(x)\}_{j=1}^M $ be eigenvalues of the matrix-function $V(x)$.
Then
 \begin{align*}
& \sum_n \lambda_n^\gamma = C_\gamma \, 
\sum_n \int_0^\infty t^{\gamma-2} (-\lambda_n + t)_- dt\\
& \le \frac{2K(\alpha)}{3\sqrt{3}} \, \frac{C_\gamma}{L_{1,1}^{\mathrm{cl}}} \, 
\int_0^\infty \int_0^{2\pi} \int_{-\infty}^\infty  t^{\gamma-2}   
\Tr\left[\left(|\xi|^2 - V(x) + t\right)_-\right]\,\frac{d\xi dx}{2\pi} \,  dt \\
& = \frac{2K(\alpha)}{3\sqrt{3}}\, \frac{C_\gamma}{L_{1,1}^{\mathrm{cl}}} \, 
\sum_{j=1}^M\, \int_0^\infty \int_0^{2\pi} \int_{-\infty}^\infty  t^{\gamma-2} 
 \Tr\left[\left(|\xi|^2 - \mu_j + t\right)_-\right]\,\frac{d\xi dx}{2\pi} \,  dt\\
&= \frac{2K(\alpha)}{3\sqrt{3}}(L_{1,1}^{\mathrm{cl}})^{-1} \, \int_0^{2\pi} 
\int_{-\infty}^\infty   \Tr \left[\left(|\xi|^2 - V(x)\right)_-^\gamma\right]\,
\frac{d\xi dx}{2\pi} \\
& = \frac 2{3\sqrt{3}}K(\alpha)  \frac{L_{\gamma,1}^{\mathrm{cl}}}{L_{1,1}^{\mathrm{cl}}}\, 
\int_0^{2\pi}  \Tr[V(x)^{1/2 + \gamma}] \, dx.
\end{align*}
\end{proof}

\subsection*{Magnetic Schr\"odinger operator in $\mathbb T^n\times \mathbb R^m$}\label{ss62}
Let us consider the eigenvalue problem for the
 Schr\"odinger
operator $\mathcal{H}$ in $L_2(\mathbb T^{d_1}\times \mathbb R^{d_2})$:
\begin{multline}\label{cyl}
\mathcal{H}\Psi= -\Delta_y\Psi+
\left(i\,\nabla_x - A(x)\right)^2 \Psi-V(x,y)\Psi=-\lambda\Psi, \\(x,y)\in 
\mathbb T^{d_1}\times \mathbb R^{d_2} ,
\end{multline}
where $\mathbb T^{d_1} = \underbrace{\mathbb S^1\times \mathbb S^1}_{d_1-times}$ is 
the standard torus of dimension $d_1$ and
$$A(x) = \left(a_1(x_1),\dots a_{d_1} (x_{d_1})\right)$$
is the magnetic vector potential in the ``diagonal'' case $a_j(x)=a_j(x_j)$.
Assume that
$$
\alpha_j = \frac{1}{2\pi}\, \int_0^{2\pi} a_j(x_j) \, dx_j \not\in\mathbb Z, \qquad 1\le j \le d_1.
$$
Then we have

\begin{theorem}\label{Th gamma} Suppose that the potential
$V(x,y)\ge0$ in \eqref{cyl} and $V\in L_{\gamma + (d_1+d_2)/2}(\mathbb T^{d_1} 
\times \mathbb R^{d_2})$.
If $\gamma \ge1/2$, then the following bound holds
for the negative eigenvalues:
\begin{equation}\label{gamma-moments}
\sum_{n}\lambda_n^{\gamma}
\le L_{\gamma, d_1+d_2} \, \int_{\mathbb T^{d_1}
\times \mathbb R^{d_2}} V^{\gamma + (d_1+d_2)/2} (x,y)\, dx dy.
\end{equation}
Here, if $d_1, d_2\ge 1$ and $1/2\le \gamma <1$, then
$$
\aligned
L_{\gamma, d_1+d_2}\le R_{\gamma, d_1+d_2} :=  2\, 
\left(\frac{2}{3\sqrt3}\right)^{d_1} 
\frac{L_{\gamma,d_1+d_2}^{\mathrm{cl}}}{(L_{1,1}^{\mathrm{cl}})^{d_1}} 
\prod_{j=1}^{d_1} K(\alpha_j)=\\=
2\, \left(\frac{\pi}{\sqrt3}\right)^{d_1} 
L_{\gamma,d_1+d_2}^{\mathrm{cl}} \prod_{j=1}^{d_1} K(\alpha_j),
\endaligned
$$
 and, if $d_1\ge1$, $d_2\ge 0$ and $\gamma\ge1$, then
$$
L_{\gamma, d_1+d_2}\le \frac12  \, R_{\gamma, d_1+d_2}.
$$
\end{theorem}
\begin{proof}As in Corollary \ref{1D gamma-moments} it is enough to prove 
this result for smooth compactly supported potentials.
We shall use the so-called ``lifting argument with respect 
to dimensions", see \cite{Lap-Weid}.

Let $x=(x_1,x')$ and $y=(y_1,y')$, where $x' \in 
\mathbb R^{d_1-1}$ and $y' \in \mathbb R^{d_2-1}$ and let
$A(x)=(a_1(x_1), A'(x'))$ . Denote
$$
-\Delta' = -(\nabla_{y'})^2, \quad -\Delta_{A} = 
\left(i\,\nabla_x - A(x)\right)^2 , \quad  -\Delta_{A'} = 
 \left(i\nabla'_{x'} - A'(x')\right)^2.
$$
By applying the result in \cite{H-L-W} on the 1/2-moments we have
\begin{align*}
&\sum_n \lambda_n^{1/2}(\mathcal H)
= \sum_n \lambda_n^{1/2} (-\partial_{y_1}^2 -\Delta_A - \Delta'  - V)  \\
&\le  \sum_n \lambda_n^{1/2} \left(-\partial_{y_1}^2 - 
\left( -\Delta_A - \Delta'  - V\right)_-\right)\\
&\le 2L_{1/2,1}^{\mathrm{cl}}\, \int_{\mathbb R}\Tr \left[ -\Delta_A - \Delta'  - V\right]_- dy_1\\
&= 2L_{1/2,1}^{\mathrm{cl}} \sum_n \int_{\mathbb R} 
\lambda_n\left( \left(i\partial_{x_1} - a_1(x_1)\right)^2
-\Delta_{A'} - \Delta' -V(x,y_1,y')\right) dy_1\\
&\le
2L_{1/2,1}^{\mathrm{cl}} \sum_n \int_{\mathbb R}
\lambda_n\Big( \left(i\partial_{x_1} - a_1(x_1)\right)^2
- \left(-\Delta_{A'}- \Delta' -V(x,y_1,y')\right)_-\Big) dy_1.
\end{align*}
Then Theorem~\ref{Th:matrS1} implies
\begin{multline*}
\sum_n \lambda_n^{1/2} (\mathcal H) \le
2L_{1/2,1}^{\mathrm{cl}}L_{1,1}^{\mathrm{cl}}
\frac{2}{3\sqrt3\, L_{1,1}^{\mathrm{cl}}}\, {K(\alpha_1)} \\
\times\int_{\mathbb R} \int_0^{2\pi}
\Tr\left[ -\Delta_{A'} - \Delta' -V(x,y_1,y') \right]_-^{3/2} dx_1 dy_1.
\end{multline*}
Now we first repeat this argument $d_1-1$ times ``splitting"
the operator $\left(i\nabla' - A'(x)\right)^2$ and using
Corollary~\ref{1D gamma-moments}. Then repeat it again $d_2-1$ times
``splitting"  the operator $\Delta'$ and using the semiclassical estimates
for the $\gamma$-Riesz means with $\gamma\ge 3/2$ for the
negative eigenvalues  of the Schr\"odinger operators with
matrix-valued potentials \cite{Lap-Weid}. Finally we obtain
\begin{multline*}
\sum_n \lambda_n^{1/2} (\mathcal H)
\le 2 \left(\prod_{l=0}^{d_1+d_2-1} L^{\mathrm{cl}}_{\gamma + l/2,1}\right) 
\left(\frac{2}{3\sqrt3 \ L_{1,1}^{\mathrm{cl}}}\right)^{d_1} \prod_{j=1}^{d_1} K(\alpha_j)\\
\times \int_{\mathbb T^{d_1}\times \mathbb R^{d_2}} V^{\gamma + (d_1+d_2)/2} (x,y)\, dx dy.
\end{multline*}
In order to prove \eqref{gamma-moments}  for the case $d_1,d_2\ge1$
and $1/2\le \gamma<1$, it remains to notice that (see~\eqref{class})
$$
\prod_{l=0}^{d_1+d_2-1} L^{\mathrm{cl}}_{\gamma + l/2,1} = L_{\gamma, d_1+d_2}^{\mathrm{cl}}.
$$
For the proof of the case $d_1\ge 1$, $d_2\ge0$ and $\gamma\ge1$
we argue similarly, but we omit the first step in the previous
argument starting directly with 1-moments.
\end{proof}

For the special cases $d_1=d_2=1$ and $d_1=2$, $d_2=0$
we state the following corollary of Theorem \ref{Th gamma}:

\begin{corollary}\label{Th:1/2-mom} Suppose that the potential
$V(x,y)\ge0$ in \eqref{cyl}. Then
 for $d_1=d_2=1$  the following bounds hold
for the $1/2$- and $1$-moments of the negative eigenvalues:
$$
\aligned
\sum_{k}\lambda_k^{1/2}&\le
\frac1{3\sqrt{3}}K(\alpha)\int_{\mathbb{R}\times\mathbb{S}^1}
V^{3/2}(x,y)dydx,\\
\sum_{k}\lambda_k&\le
\frac1{8\sqrt{3}}K(\alpha)\int_{\mathbb{R}\times\mathbb{S}^1}
V^{2}(x,y)dydx.
\endaligned
$$
For $d_1=2$, $d_2=0$ we have
$$
\sum_{k}\lambda_k\le
\frac{\pi}{24}\,K(\alpha_1)K(\alpha_2)\int_{\mathbb{T}^2}
V^{2}(x_1,x_2)dx_1dx_2.
$$
\end{corollary}

\subsection*{Acknowledgements}\label{SS:Acknow}
A.I. and S.Z.  acknowledge financial support
from the Russian Science Foundation (grant no. 14-21-00025) and M.L.'s research is
supported in part by US NSF grant DMS - 1301555.

\end{document}